\documentclass[11pt,english,a4paper]{amsart}

\usepackage[utf8]{inputenc}
\usepackage[abbrev,backrefs,nobysame]{amsrefs}
\usepackage{amsmath,graphicx}
\usepackage{amssymb}
\usepackage{blkarray}
\usepackage{amsthm}
\usepackage{enumerate}
\usepackage{tikz-cd}
\usepackage{mathrsfs}
\usepackage{dsfont}
\usepackage{pst-all}
\usepackage{amsfonts}
\usepackage{amsxtra}

\theoremstyle{definition}
\newtheorem{Def}{Definition}[section]
\newtheorem{Thm}[Def]{Theorem}
\newtheorem{Prop}[Def]{Proposition}
\newtheorem{Coro}[Def]{Corollary}
\newtheorem{Rmq}[Def]{Remark}

\newtheorem{Ex}[Def]{Example}

\newtheorem{Lem}[Def]{Lemma}

\newtheorem{Ques}[Def]{Question}

\newcommand{\fonction}[5]{\begin{array}[t]{crcl}
		#1 & #2 & \to & #3   \\  & #4 & \mapsto & #5 \end{array}}

\newcommand{\adhe}[1]{\overline{#1}}

\newcommand{\accolade}[2]{\left\{\begin{array}{#1} #2 \end{array} \right.}

\newcommand{\abs}[1]{\left| #1 \right|}

\title{
	A study of Bishop operators from the point of view of linear dynamics
}
\author{
	Vincent B\'ehani
}
\address{
	Univ. Lille, CNRS, UMR 8524 - Laboratoire Paul Painlev\'e, F-59000 Lille, France
}
\email{
	vincent.behani@univ-lille.fr
}
\date{\today}

\begin{document}
	
	\thanks{
		This work was supported in part by the project FRONT of the French National Research Agency (grant ANR-17-CE40-0021) and by the Labex CEMPI (ANR-11-LABX-0007-01). 
	}

	\thanks{
		I am grateful to Sophie Grivaux for our interesting discussions and her helpful proofreadings.
	}


	
	\begin{abstract}
		In this paper, we study the so-called Bishop operators $T _ \alpha$ on $L ^ p ([0, 1])$, with $\alpha \in (0, 1)$ and $1 < p < + \infty$, from the point of view of linear dynamics.
		We show that they are never hypercyclic nor supercyclic, and investigate extensions of these results to the case of weighted translation operators.
		We then investigate the cyclicity of the Bishop operators $T _ \alpha.$
		Building on results by Chalendar and Partington in the case where $\alpha$ is rational, we show that $T _ \alpha$ is cyclic for a dense $G _ \delta$-set of irrational $\alpha$'s, discuss cyclic functions and provide conditions in terms of convergents of $\alpha \in \mathbf R \backslash \mathbf Q$ implying that certain functions are cyclic.
	\end{abstract}

	\maketitle
		
	\section{Introduction}
	
		One of the most famous open problems in modern operator theory is the Invariant Subspace Problem in the Hilbertian setting, which can be stated as follows:
		
		Let $\mathcal H$ be a separable infinite-dimensional complex Hilbert space. For every bounded linear operator $T$ acting on $\mathcal H$, does there exist a non-trivial closed invariant subspace for $T$?

		Bishop suggested in the fifties a family of operators as counterexamples to the Invariant Subspace Problem on the complex Hilbert space $L ^ 2([0, 1])$, defined for every $\alpha \in [0, 1]$ and every $f \in L ^ 2([0, 1])$ by 
		\[
			T _ \alpha f(x) =
			x f(\{x + \alpha\})
			\quad
			\text{a.e. on }
			[0, 1],
		\]
		where $\{\cdot\}$ denotes the fractional part of a real number.
		These Bishop operators can be seen as weighted translation operators defined for every $\alpha \in [0, 1]$, every $\phi \in L ^ \infty([0, 1])$ and every $f \in L ^ 2([0, 1])$ by
		\[
			T _ {\phi, \alpha} f(x) =
			\phi(x) f(\{x + \alpha\})
			\quad
			\text{a.e. on }
			[0, 1],
		\]
		that is to say the composition of the multplication operator $M _ \phi$ and of the translation operator $U _ \alpha$ defined on $L ^ 2([0, 1])$ by
		\[
			M _ \phi f = \phi f 
			\quad
			\text{ and }
			\quad
			U _ \alpha f = 
			f(\{\cdot + \alpha\})
			\quad
			\text{for every }
			f \in L ^ 2([0, 1]).
		\]
		Even if these two types of operators are separately very well understood, the behaviour of the Bishop operators or of the weighted translation operators are still rather mysterious in the general case of irrational numbers $\alpha.$
		
		In 1965, Parrott first studied in \cite{par65} these weighted translation operators and succeeded to compute the spectrum of the Bishop operators, proving that 
		\[
			\sigma (T _ \alpha) = 
			\left\{
				z \in \mathbf C ; 
				\abs z \leq e ^ {-1}
			\right\}
			\quad
			\text{for every irrational }
			\alpha \in [0, 1].
		\]
		
		By means of a functional calculus approach, Davie proved in \cite{dav74} that whenever $\alpha$ is a non-Liouville number in $[0, 1]$, the operator $T _ \alpha$ admits a non-trivial hyperinvariant subspace, that is to say a non-trivial subspace which is invariant by every operator which commutes with $T _ \alpha.$
		Recall that an irrational number $\alpha$ is a Liouville number if there exists a sequence $(p _ n / q _ n) _ {n \geq 0}$ of rational numbers such that
		\[
			\abs{
				\alpha - \frac{p _ n}{q _ n}
			} <
			\frac 1 {q _ n ^ n}
			\quad
			\text{for every }
			n \geq 0.
		\]

		These results have been extended for instance by MacDonald in \cite{mac90} to the weighted translation operators $T _ {\phi, \alpha}.$
		However the extensions only involve new weights $\phi$ while the parameter $\alpha$ still must be a non-Liouville number.
		
		The first one who was able to pass this barrier of Liouville numbers was Flattot in 2008, who proved in \cite{fla08} that $T _ \alpha$ admits non-trivial hyperinvariant subspaces for \textit{some} Liouville numbers $\alpha.$
		
		The set of such parameters was recently further extended in \cite{cha20} by Chamizo, Gallardo-Guti\'errez, Monsalve-L\'opez and Ubis to include more Liouville numbers.
		However the existence of non-trivial closed invariant subspaces for $T _ \alpha$ for \textit{every} irrational number $\alpha$ in $[0, 1]$ is still a tantalizing open problem.
				
		Remark that an operator $T$ does not admit any non-trivial closed invariant subspace if and only if the subspace generated by the $T$-orbit $\{T ^ n x ; n \geq 0\}$ is dense in $\mathcal H$ for every $x \in \mathcal H \backslash \{0\}$, that is to say if every $x \in \mathcal H \backslash \{0\}$ is cyclic for $T.$
		
			It is thus an interesting question to investigate whether the operators $T _ \alpha$ are cyclic, that is to say admit a vector whose orbit under the action of $T _ \alpha$ spans a dense subspace of $L ^ 2([0, 1]).$
			Surprisingly enough, this natural question seems to have been considered only in the case where $\alpha \in \mathbf Q$, by Chalendar, Partington and Pozzi, see \cite{cha11} and \cite{cha10}.
			It is the main goal of this paper to study cyclicity of $T _ \alpha$ for $\alpha \in \mathbf R \backslash \mathbf Q$.
			We also consider the related hypercyclicity (existence of a dense orbit) and supercyclicity (existence of a projective dense orbit) notions.
			Since the Bishop operators and the weighted translation operators are well-defined on the spaces $L ^ p([0, 1])$ with $1 < p < + \infty$, we will carry out our study in this setting.
			
			The paper is organized as follows.
			Section 2 deals with hypercyclicity properties of Bishop and weighted translation operators.
			After recalling the main definitions, we observe (Proposition \ref{nothypercyclic}) that the operators $T _ \alpha$ cannot be hypercyclic, and that the weighted translation operators do not satisfy the so-called Hypercyclicity Criterion (Theorem \ref{HCriterion}).
			In Section 3, we will study supercyclicity properties of our operators.
			We show (Proposition \ref{supercyclic}) that weighted translation operators are not supercyclic for a large class of weights, and that no weighted translation operator satisfies the Supercyclicity Criterion (Theorem \ref{SCriterion}).
			The proofs of these results rely on the Positive Supercyclicity Criterion from \cite{leo04} as well as on a study of the point spectrum of the adjoints of these operators (Theorem \ref{wtspectrum}).
			In Section 4 we study cyclicity of the operators $T _ \alpha$ acting on $L ^p([0, 1]).$
			Building on results from \cite{cha11} which characterize cyclic functions for $T _ \alpha$ when $\alpha \in \mathbf Q$, and using Baire Category arguments, we show (Theorem \ref{comeager}) that $T _ \alpha$ is cyclic for a co-meager set of parameters $\alpha \in [0, 1].$
			We then explicit a set of irrational numbers $\alpha \in [0, 1]$ such that $T _ \alpha$ is cyclic, using the continued fraction expansion of $\alpha.$
			More precisely, for a certain class of functions $f \in L ^ p([0, 1])$, we show that $f$ is cyclic for $T _ \alpha$ as soon as the sequence $(q _ n) _ {n \geq 0}$ of denominators of the convergents of $\alpha \in \mathbf R \backslash \mathbf Q$ has infinitely many sufficiently large gaps (Theorem \ref{psi}).
			We then generalize the results from \cite{cha11} to the case of weighted translation operators with $\alpha$ rational, and generalize our cyclicity results for irrational $\alpha$ to this setting (Theorems \ref{cyclicityw}, \ref{comeagerw} and \ref{psiwt}).
			The final Section 5 presents some further remarks and open questions.
		
		\section{Hypercyclicity}
		
		The first well known dynamical property we study is called \textit{hypercyclicity}.
		We refer the reader to the books \cite{bay09} and \cite{gro11} for an in-depth study of this notion, as well as further topics in linear dynamics.
		Let $X$ be a separable infinite-dimensional complex Banach (or Fréchet) space.
		We will denote by $\mathcal B(X)$ the set of bounded linear operators on $X$.
		\begin{Def}[{\cite[Definition 2.15]{gro11}}]
			An operator $T \in \mathcal B(X)$ is said to be \textit{hypercyclic} if there exists $x \in X$ such that the $T$-orbit 
			$
				\{
					T ^ n x ; 
					n \geq 0
				\}
			$ 
			is dense in $X.$
			In this case $x$ is called a hypercyclic vector for $T.$
		\end{Def}
	
		This notion is linked to a problem similar to the Invariant Subspace Problem called the Invariant Subset Problem.
		Indeed, an operator $T \in \mathcal B(X)$ does not admit any non-trivial invariant closed subset if and only if for every $x \in X \backslash \{0\}$ the $T$-orbit 
		$
			\{
				T ^ n x ; n \geq 0
			\}
		$ 
		is dense in $X$,
		that is to say, if and only if every $x \in X \backslash \{0\}$ is hypercyclic for $T.$
		
		The following theorem shows that the set of hypercyclic vectors of an operator $T \in \mathcal B(X)$ is either empty or a dense $G _ \delta$-set in $X.$
		We recall that a $G _ \delta$-set in $X$ is a countable intersection of open sets of $X$.
		
		\begin{Thm}[Birkhoff's transitivity theorem, {\cite[Theorems 1.16 and 2.19]{gro11}}]
			An operator $T \in \mathcal B(X)$ is hypercyclic if and only if for every pair $(U, V)$ of non-empty open sets in $X$, there exists $n \geq 0$ such that $T ^ n (U) \cap V \ne \emptyset.$
			In this case, the set of hypercyclic vectors for $T$ is a dense $G _ \delta$-set in $X.$
		\end{Thm}
	
		Hypercyclicity being a well studied notion, we give here some important examples of hypercyclic operators.
		
		\begin{Ex}[{\cite[Examples 2.20, 2.21 and 2.22]{gro11}}]
			Let us consider the Hilbert space $\ell ^ 2(\mathbf N)$ and the Fréchet space $H(\mathbf C)$  of entire functions endowed with the seminorms
			$
			\|\cdot\| _ {n, \infty} 
			\colon 
			f \in H(\mathbf C) 
			\mapsto 
			\sup _ {\abs z \leq n}
			\abs{f(z)}
			$,
			$n \geq 1.$
			The operators 
			$\lambda B$ on $\ell ^ 2(\mathbf N)$, $D$ on $H(\mathbf C)$ and $T ^ {(a)}$ on $H(\mathbf C)$, defined for every $(x _ 1, x _ 2, \ldots) \in \ell ^ 2(\mathbf N)$ and every $f \in H(\mathbf C)$ by
			\[
				\lambda B(x _ 1, x _ 2, \ldots) = 
				(\lambda x _ 2, \lambda x _ 3, \ldots), 
				\quad
				D(f) = f'
				\quad
				\text{ and }
				\quad
				T ^ {(a)} f = f(\cdot + a)
			\]
			and respectively called Rolewicz's operator, MacLane's operator and Birkhoff's operator, are hypercyclic whenever $\abs \lambda > 1$ and $a\ne 0.$
		\end{Ex}
	
		A useful tool to prove the hypercyclicity of an operator is the following criterion.
		
		\begin{Thm}[Hypercyclicity Criterion, {\cite[Theorem 3.12]{gro11}}] \label{hypercyclicitycriterion}
			Let $T \in \mathcal B(X)$.
			Suppose that there exist dense subsets $X _ 0$ and $Y _ 0$ of $X$, an increasing sequence $(n _ k) _ {k \geq 1}$ of integers and a sequence $(S _ {n _ k} \colon Y _ 0 \to X) _ {k \geq 1}$ of maps such that:
			\begin{enumerate}[(i)]
				\item
					$
							T ^ {n _ k} x 
						\to 0
					$
					as 
					$
						k \to + \infty
					$ 
					for every $x \in X _ 0$;
				\item
					$
							S _ {n _ k} y 
						\to 0
					$
					as 
					$
						k \to + \infty
					$ 
					for every $y \in Y _ 0$;
				\item
					$
							T ^ {n _ k} S _ {n _ k} y
						\to y
					$
					as 
					$
						k \to + \infty
					$ 
					for every $y \in Y _ 0.$
			\end{enumerate}
			Then the operator $T$ is hypercyclic.
		\end{Thm}
	
		B\`es and Peris proved in \cite[Theorem 2.3]{bes99} that an operator $T$ on $X$ satisfies the Hypercyclicity Criterion if and only if the operator 
		$
			T \oplus T \colon 
			(x, y) \in X \oplus X 
			\mapsto 
			(Tx, Ty)
		$ 
		is hypercyclic.
		It was an open question for a long time to give examples of hypercyclic operators not satisfying the Hypercyclicity Criterion.
	
		It has been proved by Grosse-Erdmann and Peris in \cite[Theorem 3.22]{gro11} that the Hypercyclicity Criterion as stated in Theorem \ref{hypercyclicitycriterion} is actually equivalent to a similar criterion involving a partial inverse map $S$ and its iterates $(S ^ {n _ k}) _ {k \geq 0}$ instead of a sequence of maps $(S _ {n _ k}) _ {k \geq 0}.$
		
		\begin{Thm}[Gethner-Shapiro Criterion, {\cite[Theorem 2.2]{get87} and \cite[Theorem 3.10]{gro11}}]
			Let $T \in \mathcal B(X)$.
			Suppose that there exist dense subsets $X _ 0$ and $Y _ 0$ of $X$, an increasing sequence $(n _ k) _ {k \geq 1}$ of integers and a map $S \colon Y _ 0 \to Y _ 0$ such that:
			\begin{enumerate}[(i)]
				\item 
					$
							T ^ {n _ k} x
						\to 0
					$ 
					as 
					$
						k \to + \infty
					$
					for every $x \in X _ 0$;
				\item
					$
							S ^ {n _ k} y
						\to 0
					$
					as 
					$
						k \to + \infty
					$ 
					for every $y \in Y _ 0$;
				\item
					$
						T S y = y
					$ 
					for every $y \in Y _ 0.$
			\end{enumerate}
			Then the operator $T$ is hypercyclic.
		\end{Thm}
	
		These two criteria are in fact equivalent, altough the Hypercyclicity Criterion looks more general than the Gethner-Shapiro Criterion.
	
		\begin{Thm}[{\cite[Theorem 3.22]{gro11}}]\label{HCGS}
			An operator $T \in \mathcal B(X)$ satisfies the Hypercyclicity Criterion if and only if it satisfies the Gethner-Shapiro Criterion.
		\end{Thm}
	
	\subsection{Hypercyclicity of Bishop operators}
	
		First of all, it is useful to compute the iterates of $T _ \alpha.$
		For every $\alpha \in [0, 1]$, every $f \in L ^ p([0, 1])$ and every $n \geq 0$
		\[
			T _ \alpha ^ n f(x) = 
			x \{x + \alpha\} \ldots \{x + (n - 1) \alpha\} f(\{x + n \alpha\})
			\quad
			\text{a.e. on }
			[0, 1].
		\]
	
		We will be interested here in the hypercyclicity of the Bishop operators, and later on in the hypercyclicity of the weighted translation operators.
		Bishop operators are easily seen not to be hypercyclic.
		
		\begin{Prop}\label{nothypercyclic}
			For every $\alpha \in [0, 1]$, the Bishop operator $T _ \alpha$ is not hypercyclic.
		\end{Prop}
		\begin{proof}
			For every $f \in L ^ p([0, 1])$,
			\[
				\| T _ \alpha f \| _ p ^ p = 
				\int _ 0 ^ 1 
					x ^ p 
					\abs{
						f(\{x + \alpha\})
					} ^ p
				dx \leq
				\int _ 0 ^ 1
				\abs{
					f(y)
				} ^ p 
			dy
			\leq
			\| f \| _ p ^ p.				
			\]
			Thus $\|T _ \alpha \| \leq 1$, 
			so every $T _ \alpha$-orbit is bounded in $L ^ p([0, 1])$ and $T _ \alpha$ is not hypercyclic.
		\end{proof}
	
	\subsection{Hypercyclicity of weighted translation operators}
	
		Since $\| T _ {\phi, \alpha} \| = \|\phi\| _ {\infty}$ for every $\phi \in L ^ \infty ([0, 1])$ and every $\alpha \in [0, 1]$, $T _ {\phi, \alpha}$ cannot be hypercyclic if $\|\phi \| _ \infty \leq 1.$ 
		We also observe that $T _ {\phi, \alpha}$ cannot be hypercyclic in the case where $m(\{\phi = 0\}) > 0$ if $m$ denotes the Lebesgue measure on $[0, 1].$
		
		\begin{Prop}
			Let $\phi \in L ^ \infty ([0, 1])$ and let $\alpha \in [0, 1].$
			If $\| \phi \| _ \infty \leq 1$ or if $m(\{\phi = 0\}) > 0$, then the operator $T _ {\phi, \alpha}$ is not hypercyclic.			
		\end{Prop}
		\begin{proof}
			Suppose that $m(\{\phi = 0\}) > 0.$
			Since for every $f \in L ^ p([0, 1])$ and every $n \geq 0$,
			$
				T _ {\phi, \alpha} ^ n f (x) = 
				\phi(x) \phi(\{x + \alpha\}) \ldots \phi(\{x + (n - 1) \alpha\}) f(\{x + n \alpha\})
			$ 
			a.e. on $[0, 1]$, then 
			$
				\{
					\phi = 0
				\} 
				\subset
				\{
					T _ {\phi, \alpha} ^ nf = 0
				\}.
			$
			Hence the orbit $\{T _ {\phi, \alpha} ^ n f ; n \geq 0\}$ cannot be dense in $L ^ p([0, 1])$ since for every $n \geq 0$
			\[
				\|T _ {\phi, \alpha} ^ n f - \mathbf 1 _ {\{\phi = 0\}} \| _ p ^ p
				\geq
				\int _ {\{\phi = 0\}} 1 dx 
				\geq 
				m(\{\phi = 0\}) 
				> 0.
			\]
		\end{proof}
	
		In the next section, we will extend the set of weights $\phi$ such that $T _ {\phi, \alpha}$ is not hypercyclic, but we don't know if $T _ {\phi, \alpha}$ can ever be hypercyclic.
		However we will prove that it cannot satisfy the Gethner-Shapiro Criterion, and hence cannot satisfy the Hypercyclicity Criterion either. 
		
		\begin{Thm}\label{HCriterion}
			Let $\phi \in L ^ \infty ([0, 1])$ and let $\alpha \in [0, 1].$
			The operator $T _ {\phi, \alpha}$ cannot satisfy the Hypercyclicity Criterion.
		\end{Thm}
		\begin{proof}
			Suppose that $T _ {\phi, \alpha}$ satisfies the Hypercyclicity Criterion.
			Then by Theorem \ref{HCGS}, it satisfies the Gethner-Shapiro Criterion.
			Suppose also that 
			$
				m(\{
					\phi = 0
				\})
				= 0.
			$ 
			One can remark that for every $f, h \in L ^ p([0, 1])$, we have $T _ {\phi, \alpha} f = h$ if and only if 
			$
				h(x) = 
				f(\{x - \alpha\}) /
				\phi(\{x - \alpha\})
			$ 
			a.e. on $[0, 1].$
			
			Let $X _ 0$ and $Y _ 0$ be the two dense sets given by the Gethner-Shapiro Criterion, and let 
			$
				S \colon Y _ 0 \to Y _ 0
			$ 
			be the associated inverse map necessarily given by
			$
				S h (x) = 
				h(\{x - \alpha\}) /
				\phi(\{x - \alpha\})
			$ 
			a.e. on $[0, 1]$ for every $h\in Y _ 0.$
			Since $X _ 0$ and $Y _ 0$ are dense subsets of $L ^ p([0, 1])$, there exist $f \in X _ 0$ and $g \in Y _ 0$ such that 
			$
				m(\{\abs f > 1\}) \geq 3 / 4
			$ 
			and
			$
				m(\{\abs g > 1\}) \geq 3 / 4
			.$
			Otherwise one would have for every $f \in X _ 0$
			\begin{align*}
				\|f - \mathbf 2 \| _ p ^ p 
				\geq 
				\int _ {\{\abs{f} \leq 1\}}
					\abs{2 - \abs{f(x)}} ^ p
				dx
				\geq 
				m(\{\abs f \leq 1\}) 
				\geq 
				\frac 1 4,
			\end{align*}
			contradicting the density of $X _ 0$ in $L ^ p([0, 1])$ (a similar argument holds for $g \in Y _ 0$).
			Then by the Cauchy-Schwarz inequality, we have for every $k \geq 0$
			\begin{align*}
				\left(
					\int _ 0 ^ 1 
						\abs{f(\{x + n _ k \alpha\})g(x)} ^ {p / 2}
					dx 
				\right) ^ 2
				& \leq
				\int _ 0 ^ 1 
					\abs{
						\phi(x) \ldots \phi(\{x + (n _ k - 1) \alpha\}) 
						f(\{x + n _ k \alpha\})
					} ^ p
				dx \\
				& \phantom{\leq \int _ 0 ^ 1}
				\times 
				\int _ 0 ^ 1 
					\frac{
						\abs{
							g(x)
						} ^ p
					}{
						\abs{
							\phi(x) \ldots \phi(\{x + (n _ k - 1) \alpha\})
						} ^ p
					} 
				dx \\
				& \leq 
				\|T _ {\phi, \alpha} ^ {n _ k} f \| _ p ^ p 
				\int _ 0 ^ 1 
					\frac{
						\abs{
							g(\{y - n _ k \alpha\})
						} ^ p
					}{
						\abs{
							\phi(\{y - n_ k \alpha\}) \ldots \phi(\{y - \alpha\})
						} ^ p
					}
				dy \\
				& \leq
				\|T _ {\phi, \alpha} ^ {n _ k} f \| _ p ^ p 
				\|S ^ {n _ k} g \| _ p ^ p 
				\xrightarrow[k \to + \infty]{} 0.
			\end{align*}
			Hence there exists 
			a subsequence $(m _ k) _ {k \geq 0}$ of $(n _ k) _ {k \geq 0}$ such that 
			$
				f(\{x + m _ k \alpha\})g(x) \to 0
			$ 
			as $k \to + \infty$ a.e. on $[0, 1].$
			We consider now the sets $\omega = \{\abs f > 1\}$, $\Omega = \{\abs g > 1\}$ and
			$
				\Omega ' = \Omega \cap
				\{
					x \in [0, 1] ; 
					f(\{x + m _ k \alpha\}) g(x) \to 0
					\text{ as }
					k \to + \infty
				\}.
			$
			We have $m(\Omega ') = m(\Omega) \geq 3 / 4.$
			Since 
			\[
				\Omega ' \subset 
				\bigcup _ {K \geq 0} 
					\bigcap _ {k \geq K} 
						\{
							x \in [0, 1] ; 
							\{x + m _ k \alpha\} \notin \omega
						\}
			\]
			and since the sequence 
			$
				(\cap _ {k \geq K} 
					\{
						x \in [0, 1] ;
						\{x + m _ k \alpha\} \notin \omega
					\}
				) _ {K \geq 0}
			$
			is increasing, 
			it follows that the sequence 
			$
				(m(\cap _ {k \geq K}
					\{
						x \in [0, 1] ;
						\{x + m _ k \alpha\} \notin \omega
					\}
				)) _ {K \geq 0}
			$ 
			admits a limit satisfying
			\begin{align*}
				m(\Omega') 
				\leq
				\lim _ {K \to + \infty} 
					m\left(
						\bigcap _ {k \geq K} 
							\{x \in [0, 1] ; \{x + m _ k \alpha\} \notin \omega\}
					\right) 
				\leq 
				\liminf _ {K \to + \infty}
					m(R _ \alpha ^ {-m _ K} ([0, 1] \backslash \omega))
			\end{align*}
			where $R _ \alpha$ is the translation $R _ \alpha \colon x \in [0, 1] \mapsto \{x + \alpha\}$ which preserves the Lebesgue measure $m.$
			So 
			\[
				\frac 3 4
				\leq
				m(\Omega')
				\leq
				\liminf _ {K \to + \infty}
					m(
						R _ \alpha ^ {- m _ k}
							([0, 1] \backslash \omega)
					)
				\leq
				\liminf _ {K \to + \infty}
					m([0, 1] \backslash \omega)
				\leq 
				\frac 1 4,
			\]
			which is a contradiction.
		\end{proof}
	
	\section{Supercyclicity}
	
		We now consider a less restrictive dynamical property, called \textit{supercyclicity}, which has been much studied in the litterature.
		\begin{Def}[{\cite[Definition 2.16]{gro11}}]
			An operator $T \in \mathcal B(X)$ is said to be \textit{supercyclic} if there exists $x \in X$ such that the subset 
			$
				\{
					\lambda T ^ n x ;
					\lambda \in \mathbf C, 
					n \geq 0
				\}
			$ 
			is dense in $X.$
			In this case $x$ is called a supercyclic vector for $T.$
		\end{Def}
	
		The Birkhoff's Theorem admits a supercyclic version, which runs at follows.
					
		\begin{Thm}[{\cite[Theorem 1.12]{bay09}}]
			An operator $T \in \mathcal B(X)$ is supercyclic if and only if for every pair $(U, V)$ of non-empty open sets in $X$, there exist $n \geq 0$ and $\lambda \in \mathbf C$ such that $\lambda T ^ n (U) \cap V \ne \emptyset.$ 
			In this case, the set of supercyclic vectors for $T$ is a dense $G _ \delta$-set in $X.$
		\end{Thm}
	
		The so-called Supercyclicity Criterion is the most useful tool to prove that an operator is supercyclic.
		It is patterned after the Hypercyclicity Criterion.
		
		\begin{Thm}[Supercyclicity Criterion, {\cite[Lemma 3.1]{ber04}}]
			Let $T \in \mathcal B(X)$.
			Suppose that there exist dense subsets $X _ 0$ and $Y _ 0$ of $X,$ an increasing sequence $(n _ k) _ {k \geq 1}$ of integers, a sequence $(\lambda _ {n _ k}) _ {k \geq 1}$ of non-zero complex numbers and a sequence $(S _ {n _ k} \colon Y _ 0 \to X) _ {k \geq 1}$ of maps such that:
			\begin{enumerate}[(i)]
				\item
					$
							\lambda _ {n _ k} T ^ {n _ k} x 
							\to 0
					$ 
					as 
					$
						k \to + \infty
					$
					for every $x \in X _ 0$;
				\item
					$
							\lambda _ {n _ k} ^ {-1} S _ {n _ k} y 
							\to 0
					$ 
					as 
					$
						k \to + \infty
					$
					for every $y \in Y _ 0$;
				\item
					$
							T ^ {n _ k} S _ {n _ k} y 
							\to y
					$
					as 
					$
						k \to + \infty
					$
					for every $y \in Y _ 0.$
			\end{enumerate}
			Then the operator $T$ is supercyclic.
		\end{Thm}
	
		We are looking for a supercyclic version of Theorem \ref{HCriterion}.
		To do so, we will prove that the Supercyclicity Criterion is equivalent to a Gethner-Shapiro-type Supercyclicity Criterion.
		For this we will need the notion of \textit{universality} which is a generalization of hypercyclicity applied to a family of operators $(T _ n) _ {n \geq 0}$ instead of the iterates of a unique operator $T.$
		
		\begin{Def}[{\cite[Definition 1.55]{gro11}}]
			Let $Y$ be a metric space.
			A sequence $(T _ n \colon X \to Y) _ {n \geq 0}$ of continuous maps is said to be \textit{universal} if there exists $x \in X$ such that its orbit
			$
				\{
					T _ n x ;
					n \geq 0
				\}
			$ 
			under $(T _ n) _ {n \geq 0}$ is dense in $Y.$
			In this case $x$ is called a universal vector for $(T _ n) _ {n \geq 0}.$
		\end{Def}
			
		\begin{Thm}[Universality Criterion, {\cite[Theorem 1.57]{gro11}}]\label{universalitycriterion}
			Let $Y$ be a separable metric space and $(T _ n \colon X \to Y) _ {n \geq 0}$ a family of continuous maps.
			The set 
			of universal vectors for $(T _ n) _ {n \geq 0}$ is dense in $X$ if and only if for every non-empty open sets $U$ and $V$ of $X$ and $Y,$ there exists $n \geq 0$ such that $T _ n (U) \cap V \ne \emptyset.$			
			In this case the set 
			of universal vectors for $(T _ n) _ {n \geq 0}$ is a dense $G _ \delta$-set in $X.$
		\end{Thm}
	
		With some adjustements in the proof of the equivalence between the Hypercyclicity Criterion and the Gethner-Shapiro Criteron, given in \cite[Theorem 3.22]{gro11}, one can prove that the Supercyclicity  Criterion is equivalent to the following Gethner-Shapiro-type Supercyclicity Criterion.
		
		\begin{Thm}\label{GStype}
			An operator $T \in \mathcal B(X)$ satisfies the Supercyclicity Criterion if and only if the following holds true:
			there exist dense subsets $X _ 0$ and $Y _ 0$ of $X$, an increasing sequence $(n _ k) _ {k \geq 1}$ of integers, a sequence $(\lambda _ {n _ k}) _ {k \geq 1}$ of non-zero complex numbers and a map $S \colon Y _ 0 \to Y _ 0$ such that:
			\begin{enumerate}[(i)]
				\item
					$
							\lambda _ {n _ k} T ^ {n _ k} x
						\to 0
					$
					as 
					$
						k \to + \infty
					$ 
					for every $x \in X _ 0$;
				\item
					$
							\lambda _ {n _ k} ^ {-1} S ^ {n _ k} y
						\to 0
					$
					as 
					$
						k \to + \infty
					$ 
					for every $y \in Y _ 0$;
				\item
					$
						T S y = y
					$
					for every $y \in Y _ 0.$
			\end{enumerate}
		\end{Thm}
		\begin{proof}
			We only have to prove that the Supercyclicity Criterion implies the criterion given by Theorem \ref{GStype}.
			So let us suppose that $T$ satisfies the Supercyclicity Criterion. 
			Since $T$ has dense range the Mittag-Leffler's Theorem, as stated in \cite[Theorem 3.21]{gro11} with $X _ n = X$ and $f _ n = T \colon X \to X$ for every $n \in \mathbf N$, gives that the subspace
			\[
				Y = 
				\left\{
					x \in X ;
					\text{ there exists }
					(x _ n) _ {n \in \mathbf N} \in X ^ \mathbf N, 
					x = x _ 1 
					\text{ and }
					T x _ {n + 1} = x _ n
					\text{
						for every 
					}
					n \in \mathbf N
				\right\}
			\]
			is dense in $X.$
			Consider $X ^ \mathbf N$ endowed with the product topology, and its closed subspace
			\[
				\mathcal X =
				\left\{
					(x _ n) _ {n \in \mathbf N} \in X ^  \mathbf N;
					T x _ {n + 1} = x _ n
					\text{
						for every 
					}
					n \in \mathbf N
				\right\},
			\]			
			which is a separable Fréchet space endowed with the induced topology.
			We consider the operator $\mathcal T \colon \mathcal X \to \mathcal X$ defined by
			$
				\mathcal T(x _ 1, x _ 2, \ldots) =
				(Tx _ 1, T x _ 2, \ldots)
			$ 
			for every $(x _ 1, x _ 2, \ldots) \in \mathcal X$ and its inverse $\mathcal B \colon \mathcal X \to \mathcal X$ defined by 
			$
				\mathcal B(x _ 1, x _ 2, \ldots) =
				(x _ 2, x _ 3, \ldots)
			$ 
			for every $(x _ 1, x _ 2, \ldots) \in \mathcal X.$
			Let us now prove that the family $(\lambda _ {n _ k} \mathcal T ^ {n _ k}) _ {k \geq 1}$ is topologically transitive on $\mathcal X$, that is to say that for every pair $(\mathcal U _ 0, \mathcal V _ 0)$ of non-empty open sets in $\mathcal X$, there exists $k \geq 1$ such that $\lambda _ {n _ k} \mathcal T ^ {n _ k}(\mathcal U _ 0) \cap \mathcal V _ 0 \ne \emptyset.$
			Given $\mathcal U _ 0$ and $\mathcal V _ 0$, there exist $N \geq 1$ and non-empty open sets $U _ 1, \ldots, U _ N, V _ 1, \ldots, V _ N$ of $X$ such that 
			\[
				\mathcal U = 
				\{
					(x _ n) _ {n \in \mathbf N} \in \mathcal X; 
					x _ j \in U _ j
					\text{
						for every 
					}
					j \in \{1, \ldots, N\}
				\} 
				\subset
				\mathcal U _ 0
			\]
			and
			\[
				\mathcal V =
				\{
					(y _ n) _ {n \in \mathbf N} \in \mathcal X ;
					y _ j \in V _ j
					\text{ for every }
					j \in \{1, \ldots, N\}
				\}
				\subset
				\mathcal V _ 0.
			\]
			Let $x = (x _ n) _ {n \in \mathbf N} \in \mathcal U$ and $y = (y _ n) _ {n \in \mathbf N} \in \mathcal V.$
			Since $T ^ j x _ N = x _ {N - j} \in U _ {N - j}$ and $T ^ j y _ N = y _ {N - j} \in V _ {N - j}$ whenever $j \in \{1, \ldots, N\}$, there exist non-empty open neighborhoods $U _ N' \subset U _ N$ and $V _ N' \subset V _ N$ of $x _ N$ and $y _ N$ such that $T ^ j (U _ N ') \subset U _ {N - j}$ and $T ^ j(V _ N') \subset V _ {N - j}$ whenever $j \in \{1, \ldots, N\}.$
			Since $T$ satisfies the Supercyclicity Criterion, there exists $k \geq 1$ such that $\lambda _ {n _ k} T ^ {n _ k}(U _ N') \cap V _ N' \ne \emptyset$, and hence there exists a non-empty open set $U _ N'' \subset U _ N'$ such that $\lambda _ {n _ k} T ^ {n _ k}(U _ N'') \subset V _ N'.$
			Let $u _ N \in Y \cap U _ N''$, which exists by density of $Y.$
			Since $u _ N \in Y$, there exists $(u _ n) _ {n \geq N}$ such that $T u _ {n + 1} = u _ n$ whenever $n \geq N.$
			Moreover we consider $u _ j = T ^ {N - j} u _ N \in U _ j$ whenever $j \in \{1, \ldots, N - 1\}$ since $u _ N \in U _ N'' \subset U _ N'.$
			So $u = (u _ n) _ {n \in \mathbf N} \in \mathcal U$ and besides $\lambda _ {n _ k} \mathcal T ^  {n _ k} u \in \mathcal V.$
			Indeed $\lambda _ {n _ k} \mathcal T ^ {n _ k} u = (\lambda _ {n _ k} T ^ {n _ k}u _ 1, \lambda _ {n _ k} T ^ {n _ k} u _ 2, \ldots)$ and
			$
				\lambda _ {n _ k} T ^ {n _ k} u _ j 
				=
				\lambda _ {n _ k} T ^ {n _ k} T ^ {N - j} u _ N 
				=
				T ^ {N - j}(\lambda _ {n _ k} T ^ {n _ k} u _ N) 
				\in 
				T ^ {N - j}(V _ N') 
				\subset
				V _ j
			$
			whenever $j \in \{1, \ldots, N\}$ because $\lambda _ {n _ k} T ^ {n _ k} u _ N \in \lambda _ {n _ k} T ^ {n _ k}(U _ N'') \subset V _ N'.$
			Thus $\lambda _ {n _ k} \mathcal T ^ {n _ k} u \in \lambda _ {n _ k} \mathcal T ^ {n _ k}(\mathcal U) \cap \mathcal V$, which is hence non-empty.
			
			Let us prove now that there exist a dense subset $Y _ 0'$ of $X$, a map $S \colon Y _ 0' \to Y _ 0'$ and a subsequence $(m _ k) _ {k \geq 1}$ of $(n _ k) _ {k \geq 1}$ such that $TSy = y$ and 
			$
				\lambda _ {m _ k} ^ {-1} S ^ {m _ k} y 
				\to 0
			$
			as 
			$
				k \to + \infty
			$
			whenever $y \in Y _ 0'.$
			Since $(\lambda _ {n _ k} \mathcal T ^ {n _ k}) _ {k \geq 1}$ is topologically transitive, so is $(\lambda _ {n _ k} ^ {-1} \mathcal B ^ {n _ k}) _ {k \geq 1}$ and the sets 
			\[
				\{
					x \in \mathcal X ;
					\{
						\lambda _ {n _ k} \mathcal T ^ {n _ k} x ;
						k \geq 1
					\}
					\text{ is dense in } \mathcal X
				\}
			\text{
			and
			}
				\left\{
					x \in \mathcal X ;
					\left\{
						\lambda _ {n _ k} ^ {-1} \mathcal B ^ {n _ k} x ;
						k \geq 1
					\right\}
					\text{ is dense in } \mathcal X
				\right\}
			\]
			are $G _ \delta$-subsets of $\mathcal X$ by Theorem \ref{universalitycriterion}.
			So there exists $y = (y _ n) _ {n \in \mathbf N} \in \mathcal X$ such that the sets 
			\[
				\{
					\lambda _ {n _ k} \mathcal T ^ {n _ k} y ;
					k \geq 1
				\} =
				\{
					(
						\lambda _ {n _ k} T ^ {n _ k} y _ 1,
						\lambda _ {n _ k} T ^ {n _ k} y _ 2, 
						\ldots
					) ;
					k \geq 1
				\}
			\]
			and
			\[
				\left\{
					\lambda _ {n _ k} ^ {-1} \mathcal B ^ {n _ k} y ;
					k \geq 1
				\right\} =
				\left\{
					(
						\lambda _ {n _ k} ^ {-1} y _ {1 + n _ k},
						\lambda _ {n _ k} ^ {-1} y _ {2 + n _ k},
						\ldots
					) ;
					k \geq 1
				\right\}
			\]
			are dense in $\mathcal X.$
			We now consider the subset 
			$
				Y _ 0' = 
				\{
					\lambda y _ n ;
					\lambda \in \mathbf C \backslash \{0\}
					\text{ and }
					n \in \mathbf N
				\}
			$ 
			of $X$, which is dense in $X$ by density of the projection onto the first coordinate of the dense subset
			$
				\{
					\lambda _ {n _ k} ^ {-1} \mathcal B ^ {n _ k}y ;
					k \geq 1
				\}
			$ 
			in $\mathcal X.$
			Let us consider the map 
			$
				S \colon Y _ 0' \to Y _ 0 '
			$ 
			defined by 
			$
				S(\lambda y _ n) = \lambda y _ {n + 1}
			$
			for every $\lambda y _ n \in Y _ 0'$.
			It is well defined because if $\lambda y _ n = \mu y _ m$ with $n < m$, then $y _ n = (\mu / \lambda) y _ m$ and $T ^ {m - n} y _ m = y _ n = (\mu / \lambda) y _ m$ since $y \in \mathcal X.$
			Hence the subset  
			\[
				\{
					\lambda _ {n _ k} T ^ {n _ k} y _ m ;
					k \geq 1
				\} \subset
				\text{span}[T ^ k y _ m ; k \in \{0, \ldots, m - n - 1\}]
			\]
			could not be dense in $X$, which would contradict the density of the projection onto the m-th coordinate of the dense subset  
			$
				\{
					\lambda _ {n _ k} \mathcal T ^ {n _ k} y ;
					k \geq 1
				\}
			$ 
			of $\mathcal X.$
			
			Moreover $TS\lambda y _ n = \lambda T y _ {n + 1} = \lambda y _ n$ whenever $\lambda y _ n \in Y _ 0'$ because $y \in \mathcal X.$
			Since the subset
			$
				\left\{
					\lambda _ {n _ k} ^ {-1} \mathcal B ^ {n _ k} y ;
					k \geq 1
				\right\}
				=
				\left\{
					\left(
						\lambda _ {n _ k} ^ {-1} y _ {1 + n _ k},
						\lambda _ {n _ k} ^ {-1} y _ {2 + n _ k},
						\ldots
					\right)
				\right\}
			$
			is dense in $\mathcal X$, there exists a subsequence $(m _ k) _ {k \geq 1}$ of $(n _ k) _ {k \geq 1}$ such that 
			$
				\lambda _ {m _ k} ^ {-1} \mathcal B ^ {m _ k} y 
				= 
					\left(
						\lambda _ {m _ k} ^ {-1} y _ {1 + m _ k},
						\lambda _ {m _ k} ^ {-1} y _ {2 + m _ k},
						\ldots
					\right)
				\to 0
			$
			as
			$
				k\to + \infty
			$
			and	hence 
			$
					\lambda _ {m _ k} ^ {-1} S ^ {m _ k} \lambda y _ n
				=
					\lambda \lambda _ {m _ k} ^ {-1} y _ {n + m _ k}
				\to 0
			$ 
			as 
			$	
				k \to + \infty
			$
			whenever $\lambda y _ n \in Y _ 0'.$
			Eventually since $T$ satisfies the Supercyclicity Criterion, we consider the dense subset $X _ 0' = X _ 0$ of $X$ which has the property that 
			$
				\lambda _ {n _ k} T ^ {n _ k} x
				\to 0
			$
			as 
			$ 
				k \to + \infty
			$ 
			whenever $x \in X _ 0'.$
			We have thus shown that $T$ satisfies the following three properties:
			\begin{enumerate}[(i)]
				\item
					$
							\lambda _ {m _ k} T ^ {m _ k} x
						\to 0
					$
					as 
					$
						k \to + \infty
					$
					for every
					$
						x \in X _ 0';
					$
				\item 
					$
							\lambda _ {m _ k} ^ {-1} S ^ {m _ k} y
						\to 0
					$
					as 
					$
						k \to + \infty
					$
					for every
					$
						y \in  Y _ 0';
					$
				\item 
					$
						TSy = y
						\text{ for every }
						y \in Y _ 0',
					$
			\end{enumerate}
			and this terminates the proof.
		\end{proof}

		\subsection{Supercyclicity of Bishop operators}
		
		In order to prove that the Bishop operator cannot be supercyclic we will need the so-called \textit{Positive Supercyclicity Theorem}, proved by Le\'on-Saavedra and Müller in \cite[Corollary 4]{leo04}, see also \cite[Corollary 3.4]{bay09}.
		We will denote by $\sigma _ p(T)$ the set of eigenvalues of an operator $T \in \mathcal B(X).$
		
		\begin{Thm}[Positive Supercyclicity Theorem, {\cite[Corollary 3.4]{bay09}}]\label{positivesupercyclicity}
			Let $T \in \mathcal B(X)$.
			If $\sigma _ p(T ^ *) = \emptyset$, then $x \in X$ is supercyclic for $T$ if and only if 
			the set 
			$
				\{
					a T ^ n x ; a \in (0, +\infty), n \geq 0
				\}
			$
			is dense in $X.$
		\end{Thm}
	
		We will also need the following straightforward lemma.
		
		\begin{Lem}\label{density}
			If a subset $A$ is dense in $L ^ p([0, 1])$, then the set of the Lebesgue measures 
			$
				\{
					m(\{
						\Re(f) > 0
					\}) ;
					f \in A
				\}
			$
			is dense in $[0, 1].$
		\end{Lem}

		Davie proved in \cite[Theorem 2]{dav74} that the Bishop operator $T _ \alpha$ has no eigenvalue whatever the value of $\alpha \in [0, 1].$
		The same ideas are recalled and generalized by Flattot in \cite[Theorem 2.1]{fla08}.
		In fact, the approach proposed by Flattot in \cite{fla08} allows us to prove that the adjoint of the Bishop operator 
		$
			T _ \alpha ^ * 
			\in 
			\mathcal B(L ^ {p'}([0, 1])),
		$
		defined for every $f \in L ^ {p'}([0, 1])$ by
		\[
			T _ \alpha ^ * f(x) = \{x - \alpha\} f(\{x - \alpha\})
			\quad
			\text{a.e. on }
			[0, 1]
		\] 
		with $1 / p + 1 / p' = 1$, has no eigenvalue.
		To do so, we will need the following 
		lemma.
		
		\begin{Lem}[{\cite[Lemma 2.2]{fla08}}]\label{convex}
			If $f, g \colon [a, b] \to \mathbf R$ are increasing, convex, non-negative functions, then the product $fg$ also satisfies these properties.
			
			Moreover if $f(a) = 0$ then 
			$
			m(\{
			x \in [a, b] ;
			\abs{1 - f(x)} > 1 / 2
			\}) \geq (b - a) / 3.
			$
		\end{Lem}
	
		Let us now prove that the adjoint of the Bishop operator $T _ \alpha ^ *$ has no eigenvalue.
		
		\begin{Thm}\label{noeigenvalue}
			For every $\alpha \in [0, 1]$, the adjoint operator 
			$
				T _ \alpha ^ * \in 
				\mathcal B(L ^ {p'}([0, 1]))
			$ 
			has no eigenvalue.
		\end{Thm}
		\begin{proof}
			First, let us suppose that $\alpha = r / q$ is a rational number where $r$ and $q$ are coprime. 
			Let $\lambda \in \mathbf C$ and let $f \in L ^ {p'}([0, 1])$ satisfying $T _ \alpha ^ * f = \lambda f.$
			One can compute that a.e. on $[0, 1]$
			\begin{align*}
				T _ \alpha ^ {q*} f(x) & = 
				\{x - r / q\} \{x - 2 r / q\} \ldots \{x - q r / q\} f(\{x - q r / q\}) \\
				& = 
				\{x + (q - 1) r / q\} \{x + (q - 2) r / q\} \ldots x f(x).
			\end{align*}
			Since the map $k \in \{0, \ldots, q - 1\} \mapsto \{k r / q\}$ is one-to-one from the set $\{0, \ldots, q - 1\}$ onto the set $\{0, 1 / q, \ldots, (q - 1) / q\}$, then 
			$
			T _ \alpha ^ {q *} f(x) = x \{x + 1 / q\} \ldots \{x + (q - 1) / q\} f(x).
			$
			If $w$ is the function defined by $w(x) = x \{x + 1 / q\} \ldots \{x + (q - 1) / q\}$ for every $x \in [0, 1]$, the function $f$ satisfies $T _ \alpha ^ {q*} f = w f = \lambda ^ q f.$
			One can remark that $\{f \ne 0\} \subset \{w = \lambda ^ q\}$, implying that $m(\{f \ne 0\}) \leq m(\{w = \lambda ^ q\}).$
			However $w$ is $1/q$-periodic and since
			$
				w (x) = x(x + 1 / q) \ldots (x + (q - 1) / q)
			$ 
			for every $x \in [0, 1/q)$, it is strictly increasing on $[0, 1/q).$
			So the set $\{w = \lambda ^ q\}$ is finite and its Lebesgue measure satisfies $m(\{w = \lambda ^q\}) = 0.$
			Then $f = 0$, that is to say that the point spectrum of the adjoint operator is empty.
			
			Now suppose that $\alpha \in [0, 1]$ is irrational.
			The proof follows the lines of the proof of Theorem 2.1 in \cite{fla08}.
			By the Dirichlet's approximation Theorem, there exists a sequence $(p _ n / q _ n) _ {n \geq 1}$ of rational numbers satisfying $q _ n \to + \infty$ as $n \to + \infty$ such that we have
			$
			\abs{
				\alpha - p _ n / q _ n
			} \leq 1 / q _ n ^ 2
			$ 
			for every $n \geq 1.$
			
			Since $T _ \alpha ^ *$ is injective, let $\lambda \in \mathbf C \backslash \{0\}$ and let $f \in L ^ {p'} ([0, 1])$ be a function satisfying $T _ \alpha ^ * f = \lambda f.$
			Then for $n \geq 1$, we have
			$
				T _ \alpha ^ {q _ n *} f
				=
				\lambda ^ {q _ n} f,
			$ 
			so 
			$
				T _ \alpha ^ {q _ n *} f(\{x + q _ n \alpha\}) / {\lambda ^ {n _ k}}
				=
				f(\{x + q _ n \alpha\})
			$ 
			a.e. on $[0, 1]$, that is to say we have
			$
				\{x + (q _ n - 1) \alpha\}\ldots \{x + \alpha\} x f(x) / {\lambda ^ {n _ k}}
				=
				f(\{x + q _ n \alpha\}).
			$
			Hence $f(\{x + q _ n \alpha\}) = F _ n(x) f(x)$ a.e. on $[0, 1]$, where $F _ n$ is the function defined by 
			$
				F _ n(x) = 
				\{x + (q _ n - 1) \alpha\} \ldots \{x + \alpha\} x / \lambda ^ {q _ n}
			$ 
			for every $x \in [0, 1].$
			Then we have
			$
				f(x) - f(\{x + q _ n \alpha\})
				=
				(1 - F _ n (x)) f(x)
			$ 
			a.e. on $[0, 1]$, 
			and thus 
			$
				\abs{f(x)} =
				\abs{
						f(x) - f(\{x + q _ n \alpha\})
				}
				/ \abs{1 - F _ n (x)}.
			$
			
			As in the proof of Theorem 2.1 in \cite{fla08}, $(f(\{\cdot + q _ n \alpha\})) _ {n \geq 1}$ tends to $f$ in measure and one can construct a sequence $(n _ k) _ {k \geq 1}$ of positive integers such that 
			\[
				m\left(
					\bigcap _ {k = 1} ^ \infty 
						\left\{
							x \in [0, 1] ;
							\abs{
								f(x) - f(\{x + q _ {n _ k} \alpha\})
							}
							< 
							\frac 1 {2k}
						\right\}
				\right)
				\geq
				\frac 5 6.
			\]
			We will note 
			\[
			A = 
			\bigcap _ {k = 1} ^ \infty 
				\left\{
					x \in [0, 1] ; 
					\abs{
						f(x) - f(\{x + q _ {n _ k}\alpha\})
					} < \frac 1 {2k}
				\right\}
			\]
			and fix $k \geq 1.$ 
			Let $\sigma \colon \{1, \ldots, q _ {n _ k} - 1\} \to \{1, \ldots, q _ {n _ k} - 1\}$ be the unique permutation of the set $\{1, \ldots, q _ {n _ k} - 1\}$ which satisfies 
			$
				0 < 1 - \{\sigma(1) \alpha\} < \ldots < 1 - \{\sigma(q _ {n _ k} - 1) \alpha\}.
			$ 
			Then one can remark that 
			$
			\abs{F _ {n _k}} \colon x \mapsto x \{x + \alpha\} \ldots \{x + (q _ {n _ k} - 1) \alpha\} / \abs{\lambda} ^ {n _ k}
			$ 
			has a continuous extension which satisfies the conditions of Lemma \ref{convex} on each interval 
			$
				[1 - \{\sigma(j)\alpha\}, 1 - \{\sigma(j + 1) \alpha\}]
			$ 
			if $j \in \{1, \ldots, q _ {n _ k} - 2\}$ and on the intervals $[0, 1 - \{\sigma(1) \alpha\}]$ and $[1 - \{\sigma(q _ {n _ k} - 1)\alpha\}, 1]$ denoted by $I _ 0, \ldots, I _ {q _ {n _ k} - 1}.$
		
			Indeed $F _ {n _ k}(0) = F _ {n _ k} (1 - \{\sigma(j) \alpha\}) = 0$ and the map $x \mapsto  \{x + \sigma(j) \alpha\} = \{x + \{\sigma(j) \alpha\}\}$ is an increasing, non-negative, affine function on both intervals $[0, 1 - \{j \alpha\})$ and $[1 - \{j \alpha\}, 1)$ for every $j \in \{1, \ldots, q _ {n _ k} - 1\}.$
			Thus by Lemma \ref{convex}, for every $i \in \{0, \ldots, q _ {n _ k} - 1\}$
			$
				m\left(
					\left\{
						x \in I _ i ;
						\abs{
							1 - \abs{F _ {n _ k}(x)}
						}
						> 1 / 2
					\right\}
				\right)
				\geq 
				{
					\text{length}(I _ i)
				} / 3
			$
			and then we have the inequality
			$
				m\left(
					\left\{
						x \in [0, 1] ;
						\abs{1 - \abs{F _ {n _ k}(x)}} > 1 / 2
					\right\}
				\right)
				\geq 1 / 3
			$			
			by summing.
			Since $\abs{1 - F _ {n _ k}} \geq \abs{1 - \abs{F _ {n _ k}}}$, we have 
			$
				m(\{
					x \in [0, 1] ;
					\abs{1 - F _ {n _ k}(x)} > 1 / 2
				\})
				\geq 1 / 3.
			$
			
			If we set $B _ {k} = \{x \in [0, 1] ; \abs{1 - F _ {n _ k}(x)} > 1 / 2\}$ and $B = \cap _ {n = 1} ^ \infty \cup _ {k \geq n} B _ k$ then
			$
				m(B) = 
				\lim _ {n \to + \infty}
				m (
				\cup _ {k \geq n} B _ k
				)
				\geq 1 / 3
			$
			and we obtain the inequality 
			\[
				m(A \cap B) 
				= 
				1 - m([0, 1] \backslash (A \cup B)) 
				\geq
				1 - \frac 1 6 - \frac 2 3 
				\geq
				\frac 1 6.
			\]
			Moreover for every $x\in A \cap B$ and 
			every $n \geq 1$, there exists $k \geq n$ such that $x \in B _ k$ and 
			$
				\abs{f(x) - f(\{x + q _ {n _ k} \alpha\})}
				< 1 / 2k,
			$ 
			thus 
			$
				\abs{f(x)} / 2 
				\leq 
				 \abs{f(x)}\abs{1 - F _ {n _ k}(x)} < 1 / 2k,
			$ 
			which implies that 
			$
				\abs{f(x)} \leq 1 / k.
			$
			So $f(x) = 0$, and thus we have 
			$
				m(\{
					x \in [0, 1] ;
					f(x) = 0
				\}) 
				\geq 
				m(A \cap B) 
				\geq
				1 / 6.
			$
			Since $T _ \alpha ^ * f = \lambda f$, then for almost every $x$ in $[0, 1]$, 
			$f(x) = 0$ implies that 
			$
				f(\{x + \alpha\}) = 0.
			$
			Then the set $\{x \in [0, 1] ; f(x) = 0\}$ is $\alpha$-invariant, implying that 
			\[
			1 =
			m\left(
			\bigcup _ {k = 1} ^ \infty R _ \alpha ^ k (\{x \in [0, 1] ; f(x) = 0\})
			\right) \leq
			m(\{x \in [0, 1] ; f(x) = 0\})
			\]
			because $m(\{x \in [0, 1] ; f(x) = 0\}) \geq 1 / 6$ and $R _ \alpha$ is an ergodic transformation.
			So $f = 0$, that is to say, the point spectrum of the adjoint operator is empty.
		\end{proof}
	
		Knowing now that the adjoint of the operator $T _ \alpha$ has empty point spectrum, we can deduce from the Positive Supercyclicity Theorem that $T _ \alpha$ is not supercyclic.
		
		\begin{Thm}\label{notsupercyclic}
			For every $\alpha \in [0, 1]$, the Bishop operator $T _ \alpha$ is not supercyclic.
		\end{Thm}
		\begin{proof}
			Let $f \in L ^ p([0, 1]).$ 
			By Theorem \ref{positivesupercyclicity}, $f$ is supercyclic for $T _ \alpha$ if and only if the set
			$
				\{
					a T _ \alpha ^ n f ;
					a > 0 , n \geq 0
				\}
			$
			is dense in $L ^ p ([0, 1]).$
			Then by Lemma \ref{density}, the set of Lebesgue measures
			$
				\{
					m(
						\{
							\Re(a T _ \alpha ^ n f) > 0
						\}
					) ;
					a > 0, n \geq 0
				\}
			$
			is dense in $[0, 1].$
			Since 
			$
				a T _ \alpha ^ n f(x) =
				a x \{x + \alpha\} \ldots \{x + (n - 1) \alpha\} f(\{x + n \alpha\})
			$
			a.e. on $[0, 1]$, for every $a > 0$ and every $n \geq 1$,
			we have
			$
				\Re(a T _ \alpha ^ n f)(x) > 0
			$
			if and only if 
			$
				a x \{x + \alpha\} \ldots \{x + (n - 1) \alpha\} \Re(f)(\{x + n \alpha\}) > 0.
			$
			So 
			\begin{align*}
				m(\{
					\Re(aT _ \alpha ^ n f) > 0
				\}) 
				& =
				m(\{
					\Re(f)(\{\cdot + n \alpha\}) > 0
				\}) \\
				& =
				m(
					R _ \alpha ^ {-n} (\{
						\Re(f) > 0
					\})
				) \\
				& =
				m(\{
					\Re(f) > 0
				\})
			\end{align*}
			since $R _ \alpha$ preserves the measure $m.$
			Then the set of the Lebesgue measures, which is equal to the singleton
			$
				\{
					m(\{
						\Re(f) > 0
					\})
				\},
			$
			is not dense in $[0, 1].$
			Hence $f$ is not supercyclic for $T _ \alpha.$
		\end{proof}
	
	\subsection{Supercyclicity of weighted translation operators}
	
		Let $\phi\in L ^ \infty ([0, 1]),$ one can remark already that $T _ {\phi, \alpha}$ cannot be supercyclic if the measure
		$
			m(\{
				\phi = 0
			\})
		$
		is positive.
		
		\begin{Prop}
			Let $\phi$ be a function in $L ^ \infty([0, 1])$ satisfying 
			$
				m(\{
					\phi = 0
				\}) > 0
			$ 
			and let $\alpha \in [0, 1].$
			The operator $T _ {\phi, \alpha}$ is not supercyclic.
		\end{Prop}
		\begin{proof}
			For every $\lambda \in \mathbf C$, every $n \geq 0$ and every $f \in L ^ p([0, 1])$, we have the inequality
			$
				0 <
				m(\{
					\phi = 0
				\}) \leq
				m(\{
					\lambda T _ {\phi, \alpha} ^ n f = 0
				\}),
			$
			so the set 
			$
				\{
					\lambda T _ {\phi, \alpha} ^ n f ;
					\lambda \in \mathbf C,
					n \geq 0
				\}
			$ 
			cannot be dense in $L ^ p([0, 1]).$
		\end{proof}
	
		In order to generalize the result of Theorem \ref{notsupercyclic} about the non-supercyclicity to the weighted translation operators $T _ {\phi, \alpha}$, one would need the adjoint operator $T _ {\phi, \alpha} ^ * $, defined on $L ^ {p'}([0, 1])$ by
		$
			T _ {\phi, \alpha} ^ * f(x) 
			=
			\phi(\{x - \alpha\}) f(\{x - \alpha\})
		$ 
		for every $f \in L ^ {p'}([0, 1])$ and a.e. on $[0, 1]$,
		to have no eigenvalue whenever $\alpha \in [0, 1].$

		The argument of Flattot in \cite[Theorem 2.1]{fla08} proving that weighted translation operators $T _ {\phi, \alpha}$ have no eigenvalue can be adjusted to show the following proposition.
		
		\begin{Thm}\label{wtspectrum}
			Let $\phi \in L ^ \infty([0, 1])$ be an increasing convex function  such that $\phi(0) = 0.$
			For every $\alpha \in [0, 1]$, the operators $T _ {\phi, \alpha} \in \mathcal B(L ^ p([0, 1]))$ and $T _ {\phi, \alpha} ^ * \in \mathcal B(L ^ {p'}([0, 1]))$ have no eigenvalue.
		\end{Thm}
		\begin{proof}
			First let us suppose that $\alpha = r / p$ is a rational number where $r$ and $q$ are coprime.
			Let $\lambda \in \mathbf C$ and let $f \in L ^ {p'} ([0, 1])$ satisfying $T _ {\phi, \alpha} ^ * f = \lambda f.$
			Almost everywhere on $[0, 1]$
			\begin{align*}
				T _ {\phi, \alpha} ^ {q*} f(x) 
				& =
				\phi(\{x - r / q\}) \phi(\{x - 2r / q\}) \ldots \phi(\{x - q r / q\})
				f(\{x - q r / q\})\\
				& =
				\phi(\{x + (q - 1) r / q\}) \phi(\{x + (q - 2) r / q\}) \ldots \phi(x) f(x) \\
				& =
				\phi(x) \phi(\{x + 1 / q\}) \ldots \phi(\{x + (q - 1) / q\})
				f(x).
			\end{align*}
			If we define $w$ on $[0, 1]$ by
			$
				w(x) =
				\phi(x) \phi(\{x + 1 / q\}) \ldots \phi(\{x + (q - 1) / q\}),
			$ 
			the function $f$ satisfies $T _ {\phi, \alpha} ^ {q*}f = w f = \lambda ^ q f.$
			One can remark that $\{f \ne 0\} \subset \{w = \lambda ^ q\}$, implying that $m(\{f \ne 0\}) \leq m(\{w = \lambda ^ q\}).$
			However $w$ is $1 / q$-periodic and the function $w$
			is increasing on $[0, 1 / q).$
			So the set $\{w = \lambda ^ q\}$ is finite and $m(\{w = \lambda ^ q\}) = 0.$
			Then $f = 0$, that is to say, the point spectrum of $T _ {\phi, \alpha} ^ *$ is empty.
			
			Now suppose that $\alpha \in [0, 1]$ is irrational.
			By the Dirichlet's Theorem, there exists a sequence $(p _ n / q _ n) _ {n \geq 1}$ of rational numbers satisfying 
			$
					q _ n
				\to + \infty
			$
			as
			$
				n \to + \infty
			$
			such that 
			$
				\abs{
					\alpha - p _ n / q _ n
				} \leq 1 / q _ n ^ 2
			$ 
			for every $n \geq 1.$
			Since $\phi$ is increasing the weighted translation operator $T _ {\phi, \alpha} ^ *$ is injective.
			Let $\lambda \in \mathbf C \backslash \{0\}$ and let $f \in L ^ p ([0, 1])$ satisfying $T _ {\phi, \alpha} f = \lambda f.$
			Then for every $n \geq 1$,
			$
				T _ {\phi, \alpha} ^ {q _ n *} f 
				=
				\lambda ^ {q _ n} f
			$ implies that we have
			$
				f(\{x + q _ n \alpha\}) = F _ n(x) f(x)
			$ 
			for almost every $x$ in $[0, 1]$ 
			where $F _ n$ is defined on $[0, 1]$ by 
			$
				F _ n (x) 
				=
				\phi(x) \phi(\{x + \alpha\}) \ldots \phi(\{x + (q _ n - 1) \alpha\}) / \lambda ^ {q _ n}.
			$
			Then 
			$
				\abs{f(x)} =
					\abs{
						f(x) - f(\{x + q _ n \alpha\})
					}
				/
					\abs{
						1 - F _ n (x)
					}
			$
			a.e. on $[0, 1]$.
			As in the case of the Bishop operator, $(f(\{\cdot + q _ n \alpha\})) _ {n \geq 1}$ tends to $f$ in measure.
			One can also construct a sequence $(n _ k) _ {k \geq 1}$ of positive integers such that 
			\[
			 	m\left(
			 		\bigcap _ {k = 1} ^ \infty
			 			\left\{
			 				x \in [0, 1] ;
			 				\abs{
			 					f(x) - f(\{x + q _ {n _ k} \alpha\})
		 					}
	 						< \frac 1 {2k}
	 					\right\}
	 			\right)
	 			\geq \frac 5 6.
	 		\]
	 		For every $k \geq 1$ the function $\abs{F _ {n _ k}}$ admits a continuous extension that is an increasing, convex, non-negative function on each interval $[1 - \{\sigma(j) \alpha\}, 1 - \{\sigma(j + 1) \alpha\}]$ whenever $j \in \{1 , \ldots, q _ {n _ k} - 1\}$ and on the intervals $[0, 1 - \{\sigma(1)\alpha\}]$ and $[1 - \{\sigma(q _ {n _ k} - 1) \alpha\}, 1]$, where $\sigma$ is the unique permutation of the set $\{1, \ldots, q _ {n _ k} - 1\}$ which satisfies 
	 		$
	 			0 < 1 - \{\sigma(1)\alpha\} < \ldots < 1 - \{\sigma(q _ {n _ k} - 1) \alpha\} < 1,
	 		$
	 		since $\phi$ satisfies these properties.
	 		We have already observed in the proof of Theorem \ref{noeigenvalue} that these assumptions imply that
	 		$
	 			m(\{
	 				x \in [0, 1] ;
	 				f(x) = 0
	 			\}) \geq 1 / 6,
	 		$
	 		and so $f = 0$ since $\{x \in [0, 1]; f(x) = 0\}$ is invariant under the ergodic transformation $R _ \alpha.$
		\end{proof}
	
		As in the case of Bishop operators (Theorem \ref{notsupercyclic}), the non-negativity of the weight gives the following result.
		
		\begin{Prop}\label{supercyclic}
			Let $\phi \in L ^ \infty([0, 1])$ be an increasing convex function such that $\phi(0) = 0.$
			For every $\alpha \in [0, 1]$, the weighted translation operator $T _ {\phi, \alpha}$ cannot be supercyclic.
		\end{Prop}
		\begin{proof}
			Let $f \in L ^ p ([0, 1]).$
			Since $\sigma _ p (T _ {\phi, \alpha} ^ *) = \emptyset$, by Theorem \ref{positivesupercyclicity}, $f$ is supercyclic for $T _ {\phi, \alpha}$ if and only if the set 
			$
				\{
					a T _ {\phi, \alpha} ^ n f ;
					a > 0,
					n \geq 0
				\}
			$
			is dense in $L ^ p ([0, 1]).$
			Then by Lemma \ref{density}, the set 
			$
				\{
					m(\{
						\Re(aT _ {\phi, \alpha} ^ n f) > 0 
					\}) ;
					a > 0, n \geq 0
				\}
			$
			is dense in $[0, 1].$
			Since 
			$
				aT _ {\phi, \alpha} ^ n f(x) = 
				\phi(x) \phi(\{x + \alpha\}) \ldots \phi(\{x + (n - 1) \alpha\})
				f(\{x + n \alpha\})
			$
			a.e. on $[0, 1]$, for every $a > 0$ and every $n \geq 0$, then
			$
				\Re(aT _ {\phi, \alpha} ^ n f)(x) > 0 
			$ 
			if and only if 
			$
				\phi(x) \phi(\{x + \alpha\}) \ldots \phi(\{x + (n - 1) \alpha\})
				\Re(f)(\{x + n \alpha\}) > 0.
			$
			So
			\begin{align*}
				m(\{
					\Re(aT _ {\phi, \alpha} ^ n f) > 0
				\})
				=
				m(R _ \alpha ^ {-n}(\{
					\Re(f) > 0
				\})) 
				=
				m(\{
					\Re(f) > 0
				\})
			\end{align*}
			since $\phi$ is positive on $(0, 1]$ and $R _ \alpha$ preserves the measure $m.$
			Then the set of the Lebesgue measures, which is equal to the singleton
			$
				\{
					m(\{
					\Re(f)(x) > 0\})
				\},
			$
			is not dense in $[0, 1]$, and $f$ is not supercyclic for $T _ {\phi, \alpha}.$
		\end{proof}
	
		Using Theorem \ref{GStype}, we can prove that the Supercyclicity Criterion cannot be satisfied by any weighted translation operator.
		The proof proceeds in the same way as in the hypercyclic case.
		
		\begin{Thm}\label{SCriterion}
			For every $\phi \in L ^ \infty([0, 1])$ and every $\alpha \in [0, 1]$, the operator $T _ {\phi, \alpha}$ does not satisfy the Supercyclicity Criterion.
		\end{Thm}
		\begin{proof}
			Suppose that $T _ {\phi, \alpha}$ satisfies the Supercyclicity Criterion.
			Then by Theorem \ref{GStype}, it satisfies the Gethner-Shapiro-type Supercyclicity Criterion with data $X _ 0, Y _ 0$ and $S$.
			Suppose also that 
			$
				m(\{
					\phi = 0
				\})
				= 0.
			$
			Since $X _ 0$ and $Y _ 0$ are dense subsets of $L ^ p([0, 1])$, let $f \in X _ 0$ and $g \in Y _ 0$ be functions such that 
			$
				m(\{
					\abs{f} > 1
				\}) \geq 3 / 4
			$ 
			and
			$
				m(\{
					\abs{g} > 1
				\}) \geq 3 / 4.
			$ 
			Then by the Cauchy-Schwarz inequality, for every $k \geq 0$
			\begin{align*}
				\left(
					\int _ 0 ^ 1 
						\abs{
							f(\{x + n _ k \alpha\}) g(x)
						} ^ {p / 2}
					dx
				\right) ^ 2
				& \leq
				\int _ 0 ^ 1
					\abs{
						\lambda _ {n _ k}
						\phi(x) \ldots \phi(\{x + (n _ k - 1) \alpha\})
						f(\{x + n _ k \alpha\})
					} ^ p
				dx \\
				& \phantom{\leq \int _ 0 ^ 1}
				\times
				\int _ 0 ^ 1
					\frac{
						\abs{g(x)} ^ p
					}{
						\abs{
							\lambda _ {n _ k}
							\phi(x) \ldots \phi(\{x + (n _ k - 1) \alpha\})
						} ^ p
					}
				dx \\
				& \leq
				\|\lambda _ {n _ k} T _ {\phi, \alpha} ^ {n _ k} f \| _ p ^ p 
				\int _ 0 ^ 1 
					\frac{
						\abs{g(\{y - n _ k \alpha\})} ^ p
					}{
						\abs{
							\lambda _ {n _ k}
							\phi(\{y - n _ k \alpha\}) \ldots \phi(\{y - \alpha\})
						} ^ p
					}
					dy \\
				& \leq
				\|\lambda _ {n _ k}T _ {\phi, \alpha} ^ {n _ k} f \| _ p ^ p
				\left\|
					\lambda _ {n _ k} ^ {-1}
					S ^ {n _ k} g 
				\right\| _ p ^ p
				\xrightarrow[k \to + \infty]{} 0					
			\end{align*}
			with the map $S \colon Y _ 0 \to Y _ 0$ defined by 
			$
				Sg(x) = g(\{x - \alpha\}) / \phi(\{x - \alpha\})
			$ 
			a.e. on $[0, 1]$ for every $g \in Y _ 0.$
			As in the hypercyclic case if $(m _ k) _ {k \geq 0}$ is a subsequence of $(n _ k) _ {k \geq 0}$ such that 
			$
					f(\{x + m _ k \alpha\} g(x)
				\to 0
			$
			as 
			$
				k \to + \infty
			$ 
			for every $x$ in a full measure set $A$ in $[0, 1]$, 
			then the sets
			$
				\omega = 
				\{
					\abs{f} > 1
				\}
			$ 
			and 
			$
				\Omega' =
				\{
					\abs{g} > 1 
				\}
				\cap A
			$ 
			satisfy
			\[
				\frac 3 4 
				\leq
				m(\Omega')
				\leq
				m\left(
					\bigcup _ {K \geq 0}
						\bigcap _ {k \geq K}
							\{
								x \in [0, 1] ;
								\{x + m _ k \alpha\} \notin \omega
							\}
				\right)
				\leq 
				m([0, 1] \backslash \omega)
				\leq
				\frac 1 4,
			\]
			which is impossible.
		\end{proof}
	
	\section{Cyclicity}
	
		Our aim is now to investigate the cyclicity properties of the Bishop and weighted translation operators.
		Cyclicity is the less restrictive dynamical property of operators which we consider in this paper.
		We will denote by $\mathbf C[\xi]$ the set of complex polynomials.
		
		\begin{Def}[{\cite[Definition 2.16]{gro11}}]
			An operator $T \in \mathcal B(X)$ is said to be \textit{cyclic} if there exists $x \in X$ such that the linear subspace
			$
				\text{span}[T ^ n x ; n \geq 0] =
				\{
					P(T)x ; P \in \mathbf C[\xi]
				\}
			$
			is dense in $X.$
			In this case $x$ is called a cyclic vector for $T.$
		\end{Def}
	
		As in the case of hypercyclicity and supercyclicity, there exists a well known necessary and sufficient condition for the cyclicity of an operator in the case where $T ^ *$ has no eigenvalue, proved by Baire Category arguments.
		
		\begin{Prop}\label{cyclicbirkhoff}
			Let $T \in \mathcal B(X)$ such that 
			$
				\sigma _ p(T ^ *) = \emptyset
			.$ 
			The following assertions are equivalent:
			\begin{enumerate}[(i)]
				\item 
					$T$ is cyclic;
				\item \label{ii}
					For every pair $(U, V)$ of non-empty open sets in $X$, there exists $P \in \mathbf C[\xi]$ such that 
					$
						P(T)(U) \cap V \ne \emptyset.
					$
			\end{enumerate}
			In this case, the set of cyclic vectors for $T$ is a dense $G _ \delta$-set in $X.$
		\end{Prop}
		\begin{proof}
			Suppose that the operator $T$ is cyclic.
			Let $x$ be a cyclic vector for $T$ and let $U$ and $V$ be two non-empty open sets in $X.$
			There exists $P \in \mathcal C[\xi]$ such that $P(T)x \in U.$
			Since $\sigma _ p(T ^ *) = \emptyset$, the operator $P(T)$ has dense range.
			Hence one can find an open set $W$ of $X$ such that $P(T)(W) \subset V.$
			Once again by cyclicity there exists $Q \in \mathcal C[\xi]$ such that $Q(T)x \in W$ and then 
			$
				P(T)Q(T)x = Q(T)P(T)x \in Q(T)(U) \cap V \ne \emptyset.
			$ 
			
			Conversely, let us suppose that the condition (\ref{ii}) holds and let $(U _ k) _ {k \geq 1}$ be a basis of open sets in $X.$
			A vector $x \in X$ is cyclic for $T$ if and only if for every $k \geq 1$, there exists $P \in \mathbf C[\xi]$ such that $P(T) x \in U _ k$, that is to say if and only if 
			$
				x \in 
				\bigcap _ {k \geq 1}
					\bigcup _ {P \in \mathbf C[\xi]}
						P(T) ^ {-1}(U _ k).
			$
			The union $\cup _ {P \in \mathbf C[\xi]} P(T) ^ {-1}(U _ k)$ is a dense open set in $X$ for every $k \geq 1.$
			Indeed for every non-empty open set $V$ in $X$ there exists  $P \in \mathbf C[\xi]$ such that $P(T)(V) \cap U _ k \ne \emptyset$, that is to say that $P(T) ^ {-1} (U _ k) \cap V \ne \emptyset.$
			So $\cup _ {P \in \mathbf C[\xi]} P(T) ^ {-1}(U _ k) \cap V \ne \emptyset.$
			By the Baire Category Theorem, the set of cyclic vectors is a dense $G _ \delta$-set in $X$ and thus a non-empty set.
		\end{proof}
	
	\subsection{Cyclicity of Bishop operators in the rational case}
	
		The cyclicity of the Bishop operator $T _ \alpha$ has already been studied in the case $\alpha \in \mathbf Q$ by Chalendar and Partington in \cite[Section 5.4]{cha11} and generalized to multivariable Bishop operators by Chalendar, Partington and Pozzi in \cite[Section 4]{cha10}.
		
		\begin{Def}[{\cite[Definition 5.4.1]{cha11}}]
			Let $f \in L ^ p ([0, 1])$ and let $\alpha = r / q$ be a rational number such that $0 < r < q$  and $r$ and $q$ are coprime.
			We define $\Delta(f, r / q)$ a.e. on $[0, 1]$ by
			\[
				\begin{vmatrix}
					f(t) & T _ {r / q} f(t) & \cdots & T _ {r / q} ^ {q - 1} f(t) \\
					f\left(\left\{t + r / q\right\}\right) & T _ {r / q} f\left(\left\{t + r / q\right\}\right) & \cdots & T _ {r / q} ^ {q - 1} f\left(\left\{t + r / q\right\}\right) \\
					\vdots & \vdots & & \vdots \\
					f\left(\left\{t + (q - 1) r / q\right\}\right) & T _ {r / q} f\left(\left\{t + (q - 1) r / q \right\}\right) & \cdots & T _ {r / q} ^ {q - 1} f\left(\left\{t + (q - 1) r / q\right\}\right)
				\end{vmatrix}.
			\]
		\end{Def}
		
		This function is used in \cite[Theorem 5.4.4]{cha11} to give a necessary and sufficient condition for a function $f \in L ^ p([0, 1])$ to be cyclic for $T _ \alpha$, where $\alpha \in (0, 1) \cap \mathbf Q$.
		
		\begin{Thm}[{\cite[Theorem 5.4.4]{cha11}}]\label{cyclicrat}
			Let $\alpha = r / q$ be a rational number such that $0 < r < q$ and $r$ and $q$ are coprime. 
			A function $f \in L ^ p([0, 1])$ is cyclic for $T _ {r / q}$ if and only if the function $\Delta(f, r / q)$ 
			satisfies 
			$
				m(\{
					t \in [0, 1] ;
					\Delta(f, r / q)(t) = 0
				\}) = 0.
			$
		\end{Thm}
	
		One can deduce from the previous theorem a set of common cyclic vectors for the family of operators $T _ \alpha$, $\alpha \in (0,1) \cap \mathbf Q.$
		
		\begin{Thm}\label{cyclic}
			Any holomorphic function $f$ on a open neighborhood of $[0, 1]$ such that $f(0) \ne 0$ is cyclic for $T _ \alpha$ for every $\alpha \in (0,1) \cap \mathbf Q.$
		\end{Thm}
		\begin{proof}
			Let $\alpha = r / q$ be a rational number such that $0 < r < q$ and $r$ and $q$ are coprime. 
			One can remark that the function 
			$
				\abs{
					\Delta(f, r / q)
				}
			$ 
			is a $1 / q$-periodic function by permutation of the rows.
			Besides the function $\Delta(f, r / q)$ is holomorphic on $(0, 1 / q)$ and right-continuous at 0 since $\Delta(f, r / q)(t)$ is equal for every $t \in [0, 1 / q)$ to the determinant
			\[
				\begin{vmatrix}
					\left(
						\left(
							t + 
							\left\{
								\frac {i r} q
							\right\}
						\right) \ldots 
						\left(
							t + 
							\left\{
								\frac {(i + j - 1) r} q
							\right\}
						\right)
						f\left(
							t + 
							\left\{
								\frac {(i + j) r} q
							\right\}
						\right)
					\right) _ {0 \leq i, j \leq q - 1}
				\end{vmatrix}.
			\]
			Suppose that $f$ is not cyclic for $T _ {r / q}$.
			Thus  
			$
				m(\{
					t \in [0, 1] ;
					\Delta(f, r / q)(t) = 0
				\})
				> 0
			$ 
			by Theorem \ref{cyclicrat}.
			Since $\abs{\Delta(f, r / q)}$ is $1 / q$-periodic,
			$
				m(\{
					t \in [0, 1 / q) ;
					\abs{\Delta(f, r / q)(t)} = 0
				\})
				> 0.
			$
			Hence $\Delta(f, r / q)(t) = 0$ for every $t \in [0, 1 / q)$ because $\Delta(f, r / q)$ is holomorphic on $(0, 1 / q)$ and right-continuous at 0. 
			However $\Delta(f, r / q)(0)$ is equal to the determinant
			\[
				\left|
					\begin{array}{c c c c l}
						f(0) & 0 & \cdots & \cdots & 0 \\
						f(\{r / q\}) & \{r / q\} f(\{2r / q\}) & \cdots & \cdots & \{r / q\} \ldots \{(q - 1) r / q\} f(0) \\
						f(\{2 r / q\}) & \{2 r / q\} f(\{3 r / q\}) & & \reflectbox{$\ddots$} & 0 \\
						\vdots & \vdots & \reflectbox{$\ddots$} & \reflectbox{$\ddots$} & \vdots \\
						f(\{(q - 1) r / q\}) & \{(q - 1) r / q\} f(0) & 0 & \cdots & 0
					\end{array}
				\right|
			\]
			which is in turn equal to
			\[
				(-1) ^ {(q - 2)(q - 1) / 2} f(0) ^ q \prod _ {i = 0} ^ {q - 2} \{(q - 1) r / q\} \cdots \{(q - 1 - i) r / q\} \ne 0.
			\]
			This contradiction shows that $f$ is cyclic for $T _ \alpha.$
		\end{proof}
	
		\begin{Rmq}
			In particular, for every $\alpha \in (0, 1) \cap \mathbf Q$ the Bishop operator $T _ \alpha$ is cyclic  and the constant function $\mathbf 1$ is a common cyclic vector for these operators.
		\end{Rmq}
	
		\begin{Def}\label{common}
			For every subset $A$ of $[0, 1]$, we define the set
			\[
				\text{Cycl} _ A =
				\bigcap _ {\alpha \in A}
					\{
						f \in L ^ p([0, 1]) ;
						f \text{ is cyclic for } T _ \alpha
					\}.
			\]
			of common cyclic vectors for all operators $T _ \alpha$, $\alpha \in A.$
			
			In particular $\text{Cycl} _ {\mathbf Q \cap (0, 1)}$ contains any holomorphic function $f$ on an open neighborhood of $[0, 1]$ such that $f(0) \ne 0.$
		\end{Def}
	\subsection{Cyclicity of Bishop operators in the irrational case}
			
		We are now going to use the cyclicity of the Bishop operator $T _ \alpha$ for all rational numbers $\alpha \in (0, 1)$ to deduce the cyclicity of $T _ \alpha$ for some irrational numbers $\alpha \in (0, 1).$
		To do so we will first need a notion of large sets in Baire spaces, which are called \textit{co-meager sets}.
		We refer the reader to \cite[Section 8.A]{kec95} for more on this notion.
		
		\begin{Def}[{\cite[Section 8.A]{kec95}}]
			Let $X$ be a Polish space, that is to say a separable completely metrizable topological space.
			A subset $A$ of $X$ is said to be:
			\begin{enumerate}[(i)]
				\item a co-meager set in $X$ if it contains a dense $G_\delta$-set in $X$;
				\item a meager set in $X$ if the complement $X \backslash A$ is a co-meager set in $X.$
				\item satisfying the Baire property is there exists an open set $U$ of $X$ such that 
				$
					A \Delta U =
					A \backslash U \cup U \backslash A
				$ 
				is a meager set in $X.$
			\end{enumerate}
		\end{Def}
	
		Our aim is now to prove the following theorem:
		
		\begin{Thm}\label{comeager}
			The set 
			$
			\{
			\alpha \in [0, 1] ;
			T _ \alpha \text{ is cyclic}
			\}
			$
			is a co-meager set in $[0, 1].$
		\end{Thm}
		
		To do so, we recall the Kuratowski-Ulam Theorem.
		
		\begin{Thm}[Kuratowski-Ulam, {\cite[Theorem 8.41]{kec95}}]\label{Kuratowski-Ulam}
			Let $X$ and $Y$ be two Polish spaces and let $A$ be a subset of $X$ satisfying the Baire property.
			Then the following assertions are equivalent:
			\begin{enumerate}[(i)]
				\item 
					$A$ is a co-meager set in $X \times Y$;
				\item 
					$
						\{
							x \in X ;
							\{
								y \in Y ;
								(x, y) \in A
							\}
							\text{ is a co-meager set in }
							Y
						\}
						\text{ is a co-meager set in }
						X
					$;
				\item 
					$
						\{
							y \in Y ;
							\{
								x \in X ;
								(x, y) \in A
							\}
							\text{ is a co-meager set in }
							X
						\}
						\text{ is a co-meager set in }
						Y.
					$
			\end{enumerate}
		\end{Thm}
	
		One can remark that a co-meager set satisfies the Baire property, the chosen open set being the full space.
		Then we deduce from the Kuratowski-Ulam's Theorem the following corollary.
		
		\begin{Coro}\label{Kuratowski-UlamCoro}
			Let $X$ and $Y$ be two Polish spaces and let $A$ be a co-meager set in $X \times Y.$
			Then the sets
			\[
				\{
					x \in X ;
					\{
						y \in Y ;
						(x, y) \in A
					\}
					\text{ is a co-meager set in } 
					Y
				\}
			\]
			and
			\[
				\{
					y \in Y ;
					\{
						x \in X ;
						(x, y) \in A
					\}
					\text{ is a co-meager set in }
					X
				\}
			\]
			are co-meager sets in $X$ and $Y$ respectively.
		\end{Coro}
			
		To prove Theorem \ref{comeager}, our first strategy is to directly apply Corollary \ref{Kuratowski-UlamCoro} to the set 
		\[
			A = 
			\{
				(\alpha, f) \in [0, 1] \times L ^ p([0, 1]) ;
				f \text{ is cyclic for } T _ \alpha
			\},
		\]
		which does not give any information about the irrationals $\alpha$ such that $T _ \alpha$ is cyclic.
		A second more explicit strategy would be to find a $G _ \delta$-set in 
		$
			\{
				\alpha \in [0, 1] ; 
				T _ \alpha \text{ is cyclic}
			\}
		$
		(see Remark \ref{rmkcomeager}).
		
		We will now show that $A$ is a dense $G _ \delta$-set in $[0, 1] \times L ^ p([0, 1])$.
		To do so, we will need the following lemma.
		
		\begin{Lem}\label{continuous}
			For every $Q \in \mathbf C [\xi]$, the map
			\[
				\fonction{
					\phi _ Q \colon
				}{
					([0, 1], \abs\cdot) \times (L ^ p([0, 1]), \| \cdot \| _ p)
				}{
					(L ^ p([0, 1]), \| \cdot \| _ p)
				}{
					(\alpha, f)
				}{
					Q(T _ \alpha) f
				}
			\]
				is continuous.
		\end{Lem}
		\begin{proof}
			For any $M > 0$ we define $\mathcal B _ M(L ^ p([0, 1]))$ to be the set of the operators $T$ of $L ^ p([0, 1])$ such that $\|T \| \leq M.$ 
			We denote by SOT the \textit{Strong Operator Topology} on $L ^ p([0, 1])$, i.e  the topology generated by the seminorms $T \mapsto \|T f \| _ p$ for every $f \in L ^ p([0, 1])$.
			Then one can write 
			\begin{align*}
				\phi _ Q = \varphi _ {3, Q} \circ ((\varphi _ {2, Q} \circ \varphi _ 1) \oplus id)
			\end{align*}
			where
			\[
				\fonction{
					id \colon
				}{
					(L ^ p([0, 1]), \| \cdot \| _ p)
				}{
					(L ^ p([0, 1]), \| \cdot \| _ p)
				}{
					f
				}{
					f,
				}
			\]
			\[
				\fonction{
					\varphi _ 1 \colon
				}{
					([0, 1], \abs{\cdot})
				}{
					\left(
					\mathcal B _ 1(L ^ p([0, 1])), SOT
					\right)
				}{
					\alpha
				}{
					T _ \alpha,
				}
			\]
			\[
				\fonction{
					\varphi _ {2, Q} \colon
				}{
					\left(
					\mathcal B _ 1 (L ^ p([0, 1])), SOT
					\right)
				}{
					\left(
					\mathcal B _ {\| Q \| _ 1} ( L ^ p ([0, 1])), SOT
					\right)
				}{
					T
				}{
					Q(T),
				}
			\]
			\[
				\fonction{
					\varphi _ {3, Q} \colon
				}{
					(\mathcal B _ {\| Q \| _ 1} (L ^ p([0, 1])) \times SOT) 
					\times 
					(L ^ p([0, 1]), \| \cdot \| _ p)
				}{
					(L ^ p([0, 1]), \| \cdot \| _ p)
				}{
					(T, f)
				}{
					Tf.
				}
			\]
			The continuity of the maps $\varphi _ {2, Q}$ and $\varphi _ {3, Q}$ being easily proved, we only need to prove that $\varphi _ 1$ is a continuous map.
			Let $\alpha \in [0, 1]$, and let $(\alpha _ n) _ {n \geq 1}$ to be a sequence of elements of $[0, 1]$ such that $\alpha _ n \to \alpha$ as $n \to + \infty.$
			Given $f \in L ^ p([0, 1])$, we are going to show that 
			$
				\|
					T _ {\alpha _ n} f - T _ \alpha f
				\| _ p 
				\to 0
			$ as $n \to + \infty.$
			
			Let $\varepsilon > 0.$ 
			By density there exists $g \in \mathcal C([0, 1])$ such that $\|f - g \| _ p < \varepsilon.$
			For every $n \geq 1$,
			\begin{align*}
				\| T _ {\alpha _ n} f - T _ \alpha f \| _ p
				& \leq 
				\| T _ {\alpha _ n} f - T _ {\alpha _ n} g \| _ p +
				\| T _ {\alpha _ n} g - T _ \alpha g \| _ p +
				\| T _ \alpha g - T _ \alpha f \| _ p.
			\end{align*}
			On the one hand, we have for every $\beta \in [0, 1]$ that
			\[
				\| T _ \beta f - T _ \beta g \| _ p ^ p 
				= 
				\int _ 0 ^ 1 
					\abs{
						t f(\{t + \beta\}) - t g(\{t + \beta\})
					} ^ p
				dt 
				\leq
				\int _ 0 ^ 1 
					\abs{
						f(s) - g(s)
					} ^ p
				ds
				\leq 
				\| f - g \| _ p ^ p, 			
			\]
			so that 
			$
				\|T _ {\alpha _ n} f - T _ {\alpha _ n} g \| _ p + 
				\|T _ {\alpha} g - T _ \alpha f \| _ p 
				\leq 
				2 \|f - g \| _ p 
				< 2 \varepsilon.
			$
			On the other hand 
			\begin{align*}
				\| T _ {\alpha _ n} g - T _ \alpha g \| _ p ^ p 
				& =
				\int _ 0 ^ 1
					\abs{
						t g(\{t + \alpha _ n\}) - tg(\{t + \alpha\})
					} ^ p
				dt \\
				& \leq
				\int _ 0 ^ 1 
					\abs{
						g(\{t + \alpha _ n\}) - g(\{t + \alpha\})
					} ^ p
				dt \\
				& \leq
				\int _ 0 ^ 1 
					\abs{
						g(y) - g(\{y + \alpha - \alpha _ n\})
					} ^ p
				dy \\
				& \leq
				\int _ 0 ^ {1 - \{\alpha - \alpha _ n\}} 
					\abs{
						g(y) - g(y + \{\alpha - \alpha _ n\})
					} ^ p
				dy \\
				& \phantom{\leq \int _ 0 ^ {1 - \{\alpha - \alpha _ n\}}} +
				\int _ {1 - \{\alpha - \alpha _ n\}} ^ 1
					\abs{
						g(y) - g(y + \{\alpha - \alpha _ n\} - 1)
					} ^ p
				dy.
			\end{align*}
			Since $g$ is uniformly continuous on $[0, 1]$, there exists $\delta > 0$ such that 
			$
				\abs{
					g(x) - g(y)
				}
				< \varepsilon
			$ 
			whenever $x, y \in [0, 1]$ satisfy $\abs{x - y} < \delta.$
			Moreover since 
			$\alpha _ n \to \alpha$ as $n \to + \infty$, 
			there exists $n _ 0 \geq 1$ such that 
			$
				\abs{
					\alpha - \alpha _ n
				}
				< \min(\delta, \varepsilon ^ p /(2 \|g \| _ \infty) ^ p, 1)
			$ 
			for every $n \geq n _ 0.$
			Fix $n \geq n _ 0.$
			
			First, we suppose that $0 \leq \alpha - \alpha _ n.$
			Then $\{\alpha - \alpha _ n\} = \alpha - \alpha _ n = \abs{\alpha - \alpha _ n}$ and
			\begin{align*}
				\| T _ {\alpha _ n} g - T _ \alpha g \| _ p ^ p 
				& \leq
				(1 - \{\alpha - \alpha _ n\}) \varepsilon ^ p + 2 ^ p \| g \| _ \infty ^ p \{\alpha - \alpha _ n\} \\
				& \leq (1 - \abs{\alpha - \alpha _ n}) \varepsilon ^ p + 2 ^ p \|g \| _ \infty ^ p \abs{\alpha - \alpha _ n} \\
				& \leq 
				2 \varepsilon ^ p.
			\end{align*}
			
			Then, we suppose that $0 > \alpha - \alpha _ n.$
			So $\{\alpha - \alpha _ n\} = \alpha - \alpha _ n + 1 = 1 - \abs{\alpha - \alpha _ n}$ and 
			\begin{align*}
				\|T _ {\alpha _ n} g - T _ \alpha g \| _ p ^ p 
				& \leq
				(1 - \{\alpha - \alpha _ n\}) 2 ^ p \| g \| _ \infty ^ p + \{\alpha - \alpha _ n\} \varepsilon ^ p \\
				& \leq
				\abs{\alpha - \alpha _ n} 2 ^ p \| g \| _ \infty ^ p + (1 - \abs{\alpha - \alpha _ n}) \varepsilon ^ p \\
				& \leq
				2 \varepsilon ^ p.
			\end{align*}
			Eventually we get that 
			$
				\| T _ {\alpha _ n} f - T _ \alpha f \| _ p 
				\leq
				(2 + 2 ^ {1 / p}) \varepsilon
			$
			for every $n \geq n _ 0$, and thus $\varphi _ 1$ is continuous at $\alpha.$
		\end{proof}
		
		We can now show Theorem \ref{comeager}.
		
		\begin{proof}[Proof of Theorem \ref{comeager}]
			First, we prove that the set previously defined as
			$
				A =
				\{
					(\alpha, f) \in [0, 1] \times L ^ p ([0, 1]) ;
					f \text{ is cyclic for } T _ \alpha
				\}
			$ 
			is a $G _ \delta$-set in $[0, 1] \times L ^ p ([0, 1]).$
			Let $(U _ n) _ {n \geq 1}$ be a basis of open sets in $L ^ p([0, 1]).$
			For every $(\alpha, f) \in [0, 1] \times L ^ p([0, 1])$, 
			$f$ is cyclic for $T _ \alpha$ if and only if for every $n \geq 1$, there exists $Q \in \mathbf C[\xi]$ such that $Q(T _ \alpha)f \in U _ n.$
			Thus 
			\[
				A =
				\bigcap _ {n \geq 1}
					\bigcup _ {Q \in \mathbf C[\xi]}
						\phi _ Q ^ {-1}(U _ n)
			\]
			where the map $\phi _ Q$ is defined by $\phi _ Q(\alpha, f) = Q(T _ \alpha) f$ for every $(\alpha, f) \in [0, 1] \times L ^ p([0, 1])$ and is continuous by Lemma \ref{continuous}.
			So $A$ is a $G _ \delta$-set in $[0, 1] \times L ^ p([0, 1]).$
			
			We next prove that the set $A$ is dense in $[0, 1] \times L ^ p([0, 1]).$
			Let $g \in L ^ p([0, 1])$, let $\beta \in [0, 1]$, and $\varepsilon > 0.$
			There exists $\alpha \in (0, 1) \cap \mathbf Q$ such that $\abs{\beta - \alpha} < \varepsilon.$ 
			Since $T _ {\alpha}$ is cyclic by Theorem \ref{cyclic} and since $\sigma _ p ( T _ {\alpha} ^ *) = \emptyset$ by Theorem \ref{noeigenvalue}, the set of cyclic vectors for $T _ {\alpha}$ is a dense $G _ \delta$-set by Proposition \ref{cyclicbirkhoff}.
			Thus there exists $f \in L ^ p([0, 1])$ cyclic for $T _ {\alpha}$ such that $\|g - f\| _ p < \varepsilon.$
			So $(\alpha, f) \in A$ and satisfies $\abs{\beta - \alpha} < \varepsilon$ and $\|g - f\| _ p < \varepsilon.$
			Hence $A$ is a dense $G _ \delta$-set in $[0, 1] \times L ^ p([0, 1])$.
			
		Then the set
			\[
			B
			= 
			\{
			\alpha \in [0, 1] ;
			\{
			f \in L ^ p([0, 1]) ;
			f \text{ is cyclic for } T _ \alpha
			\}
			\text{ is a co-meager set in }
			L ^ p([0, 1])
			\}
			\]
			is a co-meager set in $[0, 1]$ by Corollary \ref{Kuratowski-UlamCoro}.
			Moreover,
			$
				B \subset
				\{
					\alpha \in [0, 1] ;
					T _ \alpha \text{ is cyclic}
				\},
			$ 
			and Theorem \ref{comeager} follows.
		\end{proof}
	
		We want now to explicit a set of irrational numbers $\alpha$ in $[0, 1]$ such that $T _ \alpha$ is a cyclic operator.
		More precisely, we will give a sufficient condition 
		on $\alpha$ expressed in terms of rational approximations implying that $T _ \alpha$ is cyclic.
		Some results on approximations of irrationals by rational are recalled here and can be found in \cite{bug04}.
		
		\begin{Def}[{\cite[Section 1.2]{bug04}}]
			Let $n \geq 0$ and let $(a _ k) _ {0 \leq k \leq n}$ be such that $a _ 0 \in \mathbf Z$ and $a _ k \in \mathbf N$ for every $k \in \{1, \ldots, n\}.$
			The rational number
			\[
				a _ 0 + 
				\frac 1 
				{
					a _ 1 + 
					\frac 1 
					{
						\ddots +
						\frac 1
						{
							a _ n
						}
					}
				}
			\] 
			is called a \textit{finite continued fraction} and is written $[a _ 0 ; a _ 1, \ldots, a _ n].$
		\end{Def}
	
		Let us now recall the continued fraction expansion of an irrational number.
		
		\begin{Def}[{\cite[Definition 1.2]{bug04}}]
			Let $x \in [0, 1] \backslash \mathbf Q.$
			Let $a _ 0 \in \mathbf Z$ and $\xi _ 0 \in (0, 1)$ such that 
			$
				x = a _ 0 + \xi _ 0
			.$
			Let $(a _ k) _ {k \geq 1}$ be a sequence of positive integers and $(\xi _ k) _ {k \geq 1}$ be a sequence of elements in $(0, 1)$ such that 
			$
				1 / \xi _ k = a _ {k + 1} + \xi _ {k + 1}
			$ 
			for every $k \geq 0.$
			
			The rational numbers of the sequence
			$
				(p _ n / q _ n) _ {n \geq 0}
				=
				([a _ 0 ; a _ 1 , \ldots, a _ n]) _ {n \geq 0}
			$ 
			are called the \textit{convergents} of $x.$
		\end{Def}
	
		The following result shows that these convergents give an approximation by rational numbers of an irrational number, which turns out to be optimal.
		
		\begin{Prop}[{\cite[Theorem 1.3, Corollary 1.4 and Theorem 1.4]{bug04}}]
			Let $x$ be an irrational number and $(p _ n / q _ n) _ {n \geq 0}$ its convergents.
			If
			$p _ {-1} = 1$, $q _ {-1} = 0$, $p _ 0 = a _ 0$ and $q _ 0 = 1$, then $p _ n = a _ n p _ {n - 1} +p _ {n - 2}, q _ n = a _ n q _ {n - 1} + q _ {n - 2}$ and $p _ n$ and $q _ n$ are coprime for every $n \geq 1.$
			
			Moreover the convergents $(p _ n / q _ n) _ {n \geq 0}$ converge to $x$ and satisfy
			\[
				\abs{
					x - 
				\frac {p _ n}{q _ n}
				} <
				\frac 1 {q _ n q _ {n + 1}}
				\quad
				\text{ 
					for every 
				}
				n \geq 0.
			\]
		\end{Prop} 
	
		We are now able to prove the following theorem, which gives a sufficient condition on the convergents of an irrational number $\alpha$ in $[0, 1]$ implying the cyclicity of $T _ \alpha.$
		We recall that the set $\text{Cycl} _ {\mathbf Q \cap (0, 1)}$ has been defined in Definition \ref{common} as the set of common cyclic vectors for all $T _ \alpha$, $\alpha \in \mathbf Q \cap (0, 1).$
		
		\begin{Thm}\label{psi}
			Let $f \in \text{Cycl} _ {\mathbf Q \cap (0, 1)}$.
			There exists a function $\psi _ f \colon \mathbf N \to \mathbf R _ +$ with the following property:
			if $(p _ n / q _ n) _ {n \geq 0}$ are the convergents of an irrational number $\alpha$ in $[0, 1]$ and if 
			for every $n \geq 0$ there exists $n _ 0 \geq n$ such that $q _ {n _ 0 + 1} > \psi _ f(q _ {n _ 0})$, then $f$ is cyclic for $T _ \alpha$.
		\end{Thm}
	
		\begin{proof}
			Let $(g _ j) _ {j \geq 1}$ be a dense set of functions in $L ^ p([0, 1]).$
			Let $q \geq 2$ be an integer.
			We consider the sets 
			\begin{align*}
				R _ q =
				\left\{
					\frac r q \ ;
					0 < r < q \text{ and }
					gcd(r, q) = 1
				\right\}
				\subset [0, 1]
				\text{ and }
				G _ q = 
				\{
					g _ 1, \ldots, g _ q
				\}
				\subset
				L ^ p([0, 1]).
			\end{align*}
			Since $T _ {r / q}$ is cyclic for every $r / q \in R _ q$,  for every $j \in \{1, \ldots, q\}$ there exists $Q _ {r / q, j} \in \mathbf C [\xi]$ such that
			$
				\| Q _ {r / q, j} (T _ {r / q}) f - g _ j \| _ p
				<
				1 /{2 ^ q}.
			$
		Then the finite set 
		\[
			\mathcal P _ q =
			\left\{
				Q _ {r / q, j} ;
				\left(
					r / q,
					j
				\right)
				\in R _ q \times \{1, \ldots, q\}
			\right\}
		\]
		is such that for every $j \in \{1 , \ldots, q\}$ and every $p / q \in R _ q$, there exists $Q \in \mathcal P _ q$ such that 
		$
			\|
				Q (T _ {r / q}) f - g _ j
			\| _ p
			< 
			1 / 2 ^ q.
		$
		Since the maps $\phi _ {Q, f} \colon \alpha \mapsto Q(T _ \alpha) f$ are continuous for every polynomial $Q$ in $\mathcal P _ q$ by Lemma \ref{continuous}, there exists $\delta(q) > 0$ such that for every $Q \in \mathcal P _ q$ and every $r / q \in R _ q$, 
		$
			\|
				Q(T _ \beta)f - Q(T _ {r / q})f
			\| _ p
			< 1 / 2 ^ q
		$ 
		whenever $\beta \in [0, 1]$ satisfies 
		$
			\abs{
				\beta - r / q
			} < \delta(q).
		$
		
		Let us consider the function $\psi _ f \colon \mathbf N \to \mathbf R _ +$ defined by $\psi _ f(q) = 1 / (q \delta(q))$ 
		 and an irrational number $\alpha \in (0, 1)$ whose convergents $(p _ n / q _ n) _ {n \geq 0}$ are such that
		for every $n \geq 0$, there exists $n _ 0 \geq n$ such that $q _ {n _ 0 + 1} > \psi _ f(q _ {n _ 0})$.
		Since $(q _ n) _ {n \geq 0}$ is an increasing sequence of positive integers, $q _ n \geq n$ for every $n \geq 0$.
		Fix $n \geq 0$ and $n _ 0 \geq n$ such that $q _ {n _ 0 + 1} > \psi _ f(q _ {n _ 0})$.
		Then $g _ n \in G _ {q _ n} \subset G _ {q _ {n _ 0}}$ and $p _ {n _ 0} / q _ {n _ 0} \in R _ {q _ {n _ 0}}$, so there exists $Q _ {n _ 0} \in \mathcal P _ {q _ {n _ 0}}$ such that 
		$
			\|
				Q _ {n _ 0}(T _ {p _ {n _ 0} / q _ {n _ 0}}) f - g _ n
			\| _ p 
			< 2 ^ {-q _ {n _ 0}}.
		$
		Since 
		$
			\abs{
				\alpha - p _ {n _ 0} / q _ {n _ 0}
			} 
			< 1 / (q _ {n _ 0} q _ {n _ 0 + 1})
			< \delta(q _ {n _ 0}),
		$			
		we have 
		$
			\|
				Q _ {n _ 0}(T _ \alpha) f - Q _ {n _ 0} (T _ {p _ {n _ 0} / q _ {n _ 0}}) f
			\| _ p 
			< 2 ^ {- q _ {n _ 0}}.
		$
		So 
		\begin{align*}
			\|
				Q _ {n _ 0}(T _ \alpha)f - g _ n
			\| _ p 
			& \leq
			\|
				Q _ {n _ 0}(T _ \alpha)f - Q _ {n _ 0}(T _ {p _ {n _ 0} / q _ {n _ 0}}) f
			\| _ p
			+
			\|
				Q _ {n _ 0}(T _ {p _ {n _ 0} / q _ {n _ 0}})f - g _ n
			\| _ p \\
			& <
			2 ^ {- q _ {n _ 0}} + 2 ^ {- q _ {n _ 0}} 
			\leq
			2 ^ {- (q _ {n _ 0} - 1)}
			\leq 2 ^ {-(n _ 0 - 1)}
			\leq 2 ^ {-(n - 1)}.
		\end{align*}
		Hence $f$ is cyclic for $T _ \alpha$.
		\end{proof}
	
		\begin{Rmq}\label{rmkcomeager}
			As in Theorem \ref{comeager}, Theorem \ref{psi} implies that the set of parameters $\alpha$ such that $T _ \alpha$ is cyclic is a co-meager set in $[0, 1]$. Indeed, if $q _ k (\alpha)$ denotes the denominator of the $k$-th convergent of $\alpha$, then for every
			$
				f \in 
				\text{Cycl} _ {\mathbf Q \cap (0, 1)}
			$ 
			the set
			\[
				\bigcap _ {n \geq 0}
					\bigcup _ {n _ 0 \geq n}
						\{
							\alpha \in \mathbf R \backslash \mathbf Q \cap (0, 1) ;
							q _ {n _ 0 + 1}(\alpha) > \psi _ {f}(q _ {n _ 0}(\alpha))
						\}
			\]
			is a $G _ \delta$-set in $[0, 1]$.
		\end{Rmq}
	
		\begin{Rmq}
			The proof of the previous theorem does not require the function $f$ to be cyclic for $T _ \alpha$ for \textit{every} rational number $\alpha$, but only for $\alpha \in \mathbf Q \cap (0, 1) \backslash F$, where $F$ is a finite set.
		\end{Rmq}
	
		We now move over to the study of cyclicity properties of weighted translation operators.
		
	\subsection{Cyclicity of weighted translation operators}
	
		Let $\phi$ be a function in $L ^ \infty([0, 1]).$
		Again, we observe that $T _ {\phi, \alpha}$ cannot be cyclic in the case where
		$
			m(\{
				\phi = 0
			\}) > 0.
		$
				
		\begin{Prop}
			Let $\phi \in L ^ \infty([0, 1])$ satisfying
			$
				m(\{
					\phi = 0
				\}) > 0.
			$
			For every $\alpha \in [0, 1]$, the operator $T _ {\phi, \alpha}$ is not cyclic.
		\end{Prop}
		\begin{proof}
			Let 
			$
				A =
				\{
					\phi = 0
				\}
			.$
			Then for every $f \in L ^ p([0, 1])$ and every 
			$
				P = 
				\sum _ {0 \leq k \leq d}
					a _ k \xi ^ k
				\in
				\mathbf C[\xi]
			$ 
			\[
				P(T _ {\phi, \alpha})f(x) =
				\sum _ {k = 0} ^ d
					a _ k \phi(x) \ldots \phi(\{x + (k - 1) \alpha\}) f(\{x + k \alpha\}) =
				a _ 0 f(x)
				\quad
				\text{a.e. on }
				A.
			\]
			Suppose that $f$ is cyclic for $T _ {\phi, \alpha}.$
			Then the set 
			$
				\{
					P(T _ {\phi, \alpha}) f _ {| A} ;
					P \in \mathbf C[\xi]
				\} =
				\text{span}[f _ {| A}]
			$
			is dense in $L ^ p(A)$, which is impossible.
		\end{proof}
	
		Our aim is now to extend the result of Chalendar and Partington characterizing cyclic functions for rational Bishop operators to the case of weighted translation operators $T _ {\phi, \alpha}$, with $\alpha \in \mathbf Q$ and an increasing function $\phi \in \mathcal C([0, 1], \mathbf R)$.
				
		\begin{Def}
			Let $\phi \in L ^ \infty ([0, 1])$, let $f \in L ^ p([0, 1])$ and let $\alpha = r / q$ be a rational number such that $0 < p < q$ and that $r$ and $q$ are coprime.
			We define $\Delta _ \phi(f, r / q)(t)$ a.e. on $[0, 1]$ by
			\[
			\begin{vmatrix}
				f(t) & T_ {\phi, r / q} f(t) & \cdots & T _ {\phi, r / q} ^ {q - 1} f(t) \\
				f(\{t + r / q\}) & T _ {\phi, r / q} f(\{t + r / q\}) & \cdots & T _ {\phi, r / q} ^ {q - 1} f(\{t + r / q\}) \\
				\vdots & \vdots & & \vdots \\
				f(\{t + (q - 1) r / q\}) & T _ {\phi, r / q} f (\{t + (q - 1) r / q\}) & \cdots & T _ {\phi, r / q} ^ {q - 1} f(\{t + (q - 1) r / q\})
			\end{vmatrix}.
			\]
		\end{Def}
	
		We will prove the following theorem:
		
		\begin{Thm}\label{cyclicitywt}
			Let $\alpha = r / q$ be a rational number in $(0, 1)$ such that $0 < r < q$ and $r$ and $q$ are coprime.
			Let $\phi \in \mathcal C([0, 1], \mathbf R)$ be an increasing function.
			A function $f \in L ^ p([0, 1])$ is cyclic for $T _ {\phi, \alpha}$ if and only if the function $\Delta _ \phi(f, r / q)$ satisfies 
			$
				m(\{
					t \in [0, 1] ;
					\Delta _ \phi(f, r / q)(t) = 0
				\}) = 0.
			$
		\end{Thm}
	
		The proof of Theorem \ref{cyclicitywt} follows the lines of the proof of Theorem 5.4.4 in \cite{cha11}.
		The first step in the proof of Theorem \ref{cyclicitywt} is the following lemma:
	
		\begin{Lem}\label{decompositionwt}
			Let $\phi \in \mathcal C([0, 1], \mathbf R)$ be an increasing function and let $\alpha = r / q$ be a rational number such that $0 < r < q$ and that $r$ and $q$ are coprime.
			Let $f \in L ^ p ([0, 1])$, let $n \geq 1$ and let $h \in L ^ \infty([0, 1])$ vanishing on the set
			\begin{align*}
				\Omega _ {n, f, \phi} =
				&
				\{
					t \in [0, 1] ;
					\abs{
						\Delta _ \phi(f, r / q)(t) 
					} < 1 / n
				\} \cup \\
				& 
				\bigcup _ {
					0 \leq i, j \leq q - 1
				}
					\{
						t \in [0, 1] ;
						\abs{
							T _ {\phi, r / q} ^ j f(\{t + i r / q\})
						} > n
					\}.
			\end{align*}
			There exist $1 / q$-periodic functions $h _ 0, \ldots, h _ {q - 1} \in L ^ \infty([0, 1])$ such that 
			$
				h = \sum _ {j = 0} ^ {q - 1} h _ j T _ {\phi, r / q} ^ j f.
			$
		\end{Lem}
		\begin{proof}
			The proof of Lemma \ref{decompositionwt} is patterned after the proof of Lemma 5.4.2 in \cite{cha11}.
			A decomposition of $h$ of the form $h = h _ 0 f + \ldots + h _ {q - 1} f$ is by $1 / q$-periodicity equivalent to 
			\begin{align*}
				\accolade{r c c c c c c}{
					h(t) & = &
					h _ 0(t) f(t) & + & \ldots & + & h _ {q - 1}(t) T _ {\phi, r / q} ^ {q - 1} f (t)\\
					h(t + 1 / q) & = &
					h _ 0(t) f(t + 1 / q) & + & \ldots & + & h _ {q - 1}(t) T _ {\phi, r / q} ^ {q - 1} f (t) \\
					& \vdots & & & & & \\
					h(t + (q - 1) / q) 
					& = &
					h _ 0 (t) f(t + (q - 1) / q) & + & \ldots & + & h _ {q - 1} (t) T _ {\phi, r / q} ^ {q - 1} f (t + (q - 1) / q)
				}
			\end{align*}
			a.e. on $[0, 1 / q)$, that is to say to equivalent to the system
			\begin{align*}
				\begin{pmatrix}
					h(t) \\
					\vdots \\
					h(t + (q - 1) / q)
				\end{pmatrix}
				=
				\begin{pmatrix}
					f(t) & \ldots & T _ {\phi, r / q} ^ {q - 1} f(t) \\
					\vdots & & \vdots \\
					f(t + (q - 1) / q) & \ldots & T _ {\phi, r / q} ^ {q - 1} f(t + (q - 1) / q)
				\end{pmatrix}
				\begin{pmatrix}
					h _ 0(t) \\
					\vdots \\
					h _ {q - 1}(t)
				\end{pmatrix}
			\end{align*}
			whose determinant is $\pm \Delta _ \phi(f, r / q)(t)$ by permutation of the rows since $r$ and $q$ are coprime.
			On the one hand there exist such solutions if $t \in [0, 1] \backslash \Omega _ {n, f, \phi}$ because $\abs{\Delta _ \phi(f, r / q)(t)} \geq 1 / n$.
			On the other hand it suffices to set $h _ 0(t) = \ldots = h _ {q - 1}(t) = 0$ whenever $t \in \Omega _ {n, f, \phi}.$ 
			Such functions $h _ 0, \ldots, h _ {q - 1}$ are then bounded by the definition of the set $\Omega _ {n, f, \phi}.$
		\end{proof}
	
		We will also need the following density lemma, which is an analogue of Lemma 5.4.3 in \cite{cha11}.
		
		\begin{Lem}\label{density1/q}
			Let $\phi \in \mathcal C([0, 1], \mathbf R)$ be an increasing function and let $F$ be a function in $L ^ p([0, 1 / q])$ such that 
			$
				m(\{
					t \in [0, 1 / q] ;
					F(t) = 0
				\}) = 0.
			$ 
			The set $\{Q(w) F ; Q \in \mathbf C[\xi]\}$ is dense in $L ^ p([0, 1 / q])$ where $w$ is the function defined on $[0, 1 / q)$ by
			$
				w(x) = 
				\phi(x) \phi(x + 1/q) \ldots \phi(x + (q - 1) / q).
			$
		\end{Lem}
		\begin{proof}
			Suppose that $1 / p + 1 / p' = 1$ and that $G$ is a function in $L ^ {p'}([0, 1 / q])$ such that 
			\[
				\int _ 0 ^ {1 / q}
					Q(w(t))F(t)\adhe{G(t)}
				dt = 0 
				\quad
				\text{for every } Q \in \mathbf C[\xi].
			\]
			Since $\phi$ is continuous on $[0, 1]$, the function $w$ admits a continuous extension on $[0, 1 / q]$, denoted by $w _ 0.$
			Then 
			\[
				\int _ 0 ^ {1 / q}
					Q(w _ 0(t)) F(t) \adhe{G(t)}
				dt = 0
				\quad
				\text{for every } Q \in \mathbf C[\xi].
			\]
			Since $\phi \in \mathcal C([0, 1], \mathbf R)$ is an increasing function, the algebra $\{Q(w _ 0) ; Q \in \mathbf C[\xi]\}$ in $\mathcal C([0, 1 / q])$ separates points, contains the constant functions and is closed under complex conjugation. 
			Then this algebra is dense in $(\mathcal C([0, 1 / q]), \|\cdot\| _ \infty)$ by the Stone-Weierstrass's Theorem.
			Therefore 
			\[
				\int _ 0 ^ {1 / q}
					f(t) F(t) \adhe{G(t)}
				dt
				= 0
				\quad
				\text{for every }
				f \in \mathcal C([0, 1/q]).
			\]
			Thus the function $F \adhe{G}$ in $L ^ 1 ([0, 1 / q])$ is the constant function equal to zero and so is the function $G$ in $L ^{p'}([0, 1 / q]).$
		\end{proof}
	
		With these tools at hand, we can now prove Theorem \ref{cyclicitywt}.
		\begin{proof}[Proof of Theorem \ref{cyclicitywt}]
			Assume that $f$ satisfies 
			$
				m(\{
					t \in [0, 1] ;
					\Delta _ \phi(f, r / q)(t) = 0
				\})
				= 0.
			$ 
			The union 
			\[
				A = 
				\bigcup _ {n \in \mathbf N}
					\left\{
						h \in L ^ \infty([0, 1]) ;
						h = 0 
						\text{ 
							on 
						}
						\Omega _ {n, f, \phi}
					\right\}
			\]
			is dense in $L ^ p([0, 1])$ since 
			$
				m(\cap _ {n \in \mathbf N}
					\Omega _ {n, f, \phi}
				) = 0
			.$
			Fix $\varepsilon > 0$ and 
			$
				h = 
				\sum _ {j = 0} ^ {q - 1}
					h _ j T _ {\phi, \alpha} ^ j f
				\in A
			.$
			In order to prove that there exists $Q \in \mathbf C[\xi]$ such that 
			$
				\|
					Q(T _ {\phi, \alpha}) f -
					h
				\| _ p
				< \varepsilon,
			$
			one can remark that
			\[
				Q(T _ {\phi, \alpha}) f 
				=
				\sum _ {k = 0} ^ d
					a _ k T _ {\phi, \alpha} ^ k f
				=
				\sum _ {j = 0} ^ {q - 1}
					\sum _ {n = 0} ^ m
						a _ {n q + j} T _ {\phi, \alpha} ^ {nq + j} f
				=
				\sum _ {j = 0} ^ {q - 1}
					\sum _ {n = 0} ^ m
						a _ {n q + j} w ^ n T _ {\phi, \alpha} ^ j f
				=
				\sum _ {j = 0} ^ {q - 1}
				Q _ j(w)T _ {\phi, \alpha} ^ j f
			\]
			for every 
			$
				Q =
				\sum _ {k = 0} ^ d
					a _ k \xi ^ k
				\in \mathbf C[\xi]
			$ since $T _ {\phi, \alpha} ^ q f = w f$ and by $1 / q$-periodicity of $w.$
			Therefore, in order to prove that there exists $Q \in \mathbf C[\xi]$ such that 
			$
				\|
					Q(T _ {\phi, \alpha}) f - \sum _ {j = 0} ^ {q - 1} h _ j T _ {\phi, \alpha} ^ j f
				\| _ p 
				< \varepsilon,
			$
			it suffices to show that
			for every $j \in \{0, \ldots, q - 1\}$, there exists $Q _ j \in \mathbf C[\xi]$ such that 
			$
				\|
					Q _ j(w)(T _ {\phi, \alpha}) - h _ j T _ {\phi, \alpha} ^ j f
				\| _ p
				< \varepsilon / q
			$ 
			and then the polynomial $Q$ defined by 
			$
				Q =
				\sum _ {j = 0} ^ {q - 1}
					Q _ j(\xi ^ q) \xi ^ j
			$ 
			will satisfy 
			$
				\|
					Q (T _ {\phi, \alpha})f - h
				\| _ p
				< \varepsilon.
			$
			Let $j \in \{0, \ldots, q - 1\}$ and consider the function
			\[
				F _ j \colon t \in [0, 1 / q] \mapsto
				\abs{T _ {\phi, \alpha} ^ j f(t)} + \abs{T _ {\phi, \alpha} ^ jf(\{t + 1 / q\})} + \ldots + \abs{T _ {\phi, \alpha} ^ j f(\{t + (q - 1) / q\})}.
			\]	
			For every $t \in [0, 1 / q]$, if $F _ j(t) = 0$ then $\Delta _ \phi (f, r / q)(t) = 0$ because its $j$-th column would be zero.
			So 
			$
				m(\{
					t \in [0, 1 / q] ;
					F _ j(t) = 0
				\})
				= 0
			$ 
			since
			$
				m(\{
					t \in [0, 1] ;
					\Delta _ \phi(f, r / q)(t) = 0
				\})
				= 0.
			$
			By Lemma \ref{density1/q} there exists $Q _ j \in \mathbf C[\xi]$ such that 
			$
				\|
					Q _ j(w) F _ j - h _ jF _ j
				\| _ {[0, 1 / q], p} 
				< \varepsilon / q.
			$
			However since $p > 1$ and by $1/q$-periodicity of $Q _ j(w) - h _ j$
			\begin{align*}
				\|
					(Q _ j(w) - h _ j)F _ j
				\| _ {[0, 1 / q], ^p} ^ p
				& =
					\int _ 0 ^ {1 / q}
						\abs{
							Q _ j(w(t)) - h _ j(t)
						} ^ p 
						\left(
							\sum _ {k = 0} ^ {q - 1}
								\abs{T _ {\phi, \alpha} ^ j f(t + k / q)}
						\right) ^ p
					dt
				\\
				& \geq 
				\sum _ {k = 0} ^ {q - 1}
					\int _ 0 ^ {1 / q}
						\abs{
							Q _ j(w(t)) - h _ j(t)
						} ^ p
						\abs{
							T _ {\phi, \alpha} ^ j f(t + k / q))
						} ^ p
					dt
				\\
				& \geq
					\int _ 0 ^ 1
						\abs{
							(Q _ j(w(t)) - h _ j(t)) T _ {\phi, \alpha} ^ j f(t)
						} ^ p
					dt
				\\
				& \geq
				\| 
					Q _ j(w)T _ {\phi, \alpha} ^ j f - h _ j T _ {\phi, \alpha} ^ j f
				\| _ {p} ^ p.
			\end{align*}
			Then
			$
				\| Q(T _ {\phi, \alpha}) f - h \| _ p
				\leq
				\sum _ {j = 0} ^ {q - 1}
					\|
						Q _ j(w) T _ {\phi, \alpha} ^ j f - h _ j T _ {\phi, \alpha} ^ j f
					\| _ p
				<
				\sum _ {j = 0} ^ {q - 1}
					\varepsilon / q
				\leq \varepsilon,				
			$
			so $f$ is cyclic for $T _ {\phi, \alpha}.$
			
			Conversely, let us now assume that the set
			$
				A = 
				\{
					t \in [0, 1] ;
					\Delta _ \phi(f, r / q)(t) = 0
				\}
			$ 
			satisfies $m(A) > 0.$
			Then there exist functions $a _ 0, \ldots, a _ {q - 1}$ on $[0, 1]$, not all zero, such that 
			\[
				\begin{pmatrix}
					f(t) & \ldots & f(\{t + (q - 1) r/ q\}) \\
					T _ {\phi, \alpha} f(t) & \ldots & T _ {\phi, \alpha} f(\{t + (q - 1) r/ q\}) \\
					\vdots & & \vdots \\
					T _ {\phi, \alpha} ^ {q - 1} f(t) & \ldots & T _ {\phi, \alpha} ^ {q - 1} f(\{t + (q - 1)r / q\})
				\end{pmatrix}
				\begin{pmatrix}
					a _ 0(t) \\
					a _ 1(t) \\
					\vdots\\
					a _ {q - 1}(t)
				\end{pmatrix}
				=
				\begin{pmatrix}
					0 \\
					0 \\
					\vdots \\
					0
				\end{pmatrix}
			\]
			a.e. on $A$, that is to say
			$
				\sum _ {j = 0} ^ {q - 1}
					a _ j(t) T _ {\phi, \alpha} ^ i f(\{t + jr / q\})
				= 0
			$
			for every $i \in \{0, \ldots, q - 1\}.$
			If $\varphi = Q(T _ {\phi, \alpha}) f$ with $Q = \sum _ {k = 0} ^ d b _ k \xi ^ k \in \mathbf C[\xi]$, then
			\begin{align*}
				\sum _ {j = 0} ^ {q - 1} a _ j(t) \varphi(\{t + jr / q\})
				& =
				\sum _ {k = 0} ^ d
					b _ k
					\sum _ {j = 0} ^ {q - 1}
						a _ j(t) T _ {\phi, \alpha} ^ k f(\{t + jr / q\})\\
				& =
				\sum _ {n = 0} ^ m
					\sum _ {i = 0} ^ {q - 1}
						b _ {nq + i}
						\sum _ {j = 0} ^ {q - 1}
							a _ j(t) T _ {\phi, \alpha} ^ {nq + i} f(\{t + jr / q\})\\
				& =
				\sum _ {n = 0} ^ m
					\sum _ {i = 0} ^ {q - 1}
						b _ {nq + i} w ^ n(t)
						\sum _ {j = 0} ^ {q - 1}
							a _ j(t) T _ {\phi, \alpha} ^ i f(\{t + jr / q\}) \\
				& = 0
			\end{align*}
			a.e. on $A$ since $T _ {\phi, \alpha} ^ {nq + i} f = w ^ n T _ {\phi, \alpha} ^ i f$ and by $1/q$-periodicity of $w.$
			Therefore 
			$
				\sum _ {j = 0} ^ {q - 1}
					a _ j(t)\varphi(\{t + j r / q\})
				= 0
			$
			whenever $\varphi$ is a function in the closed subspace generated by 
			$
				\{
					T _ {\phi, \alpha} ^ n f ;
					n \geq 0
				\}.
			$ 
			So the subspace $\{Q(T _ {\phi, \alpha})f ; Q \in \mathbf C[\xi]\}$ is not dense in $L ^ p([0, 1])$, that is to say, $f$ is not cyclic for $T _ {\phi, \alpha}.$
		\end{proof}
		
		Hence if $\phi$ is a holomorphic function, 
		Theorem \ref{cyclic} still holds for the weighted translation operators.
		
		\begin{Thm}\label{cyclicityw}
			Let $\phi \in \mathcal C([0, 1], \mathbf R)$ be an increasing function satisfying $\phi(0) = 0$ and which is holomorphic on an open neighborhood of $[0, 1].$
			Any holomorphic function $f$ on an open neighborhood of $[0, 1]$ such that $f(0) \ne 0$ is cyclic for $T _ {\phi, \alpha}$ for every $\alpha \in (0, 1) \cap \mathbf Q$.			
		\end{Thm}
		\begin{proof}
			Let $\alpha = r / q$ be a rational number such that $0 < r < q$ and $r$ and $q$ are coprime.
			It suffices to show that 
			$
				m(\{
					t \in [0, 1] ;
					\Delta _ \phi(f, r / q)(t) = 0
				\})
				= 0
			$
			by Theorem \ref{cyclicitywt}.
			The function $\abs{\Delta _ \phi(f, r / q)}$ is $1/q$-periodic on $[0, 1]$ and $\Delta _ \phi(f, r / q)$ is holomorphic on $(0, 1 / q)$, since $\phi$ is holomorphic, and right-continuous at 0.
			Indeed $\Delta _ \phi(f, r / q)(t)$ is equal to the determinant
			\[
				\begin{vmatrix}
					\left(
						\phi\left(
							t + 
							\left\{
								\frac {ir} q
							\right\}
						\right)
						\ldots 
						\phi\left(
							t + 
							\left\{
								\frac{(i + j - 1) r} q
							\right\}
						\right)
						f\left(
							t + 
							\left\{
								\frac{(i + j) r} q
							\right\}
						\right)
					\right) _ {0 \leq i, j \leq q - 1}
				\end{vmatrix}
			\]
			whenever $t \in [0, 1 / q).$
			Then as in the case of the Bishop operators,
			$\Delta _ \phi (f, r / q)(t) = 0$ for every $t \in [0, 1 / q)$ if $f$ is not cyclic for $T _ {\phi, \alpha}$.
			However $\Delta _ \phi(f, r / q)(0)$ is equal to the determinant
			\[
				\left|\begin{array}{c c c c l}
					f(0) & 0 & \ldots & \ldots & 0 \\
					f(\{r / q\}) & \phi(\{r / q\}) f(\{2r / q\}) & \ldots & \ldots & \phi(\{r / q\}) \ldots \phi(\{(q - 1) r / q\})f(0) \\
					f(\{2r/q\}) & \phi(\{2r / q\}) f(\{3r / q\}) & & \reflectbox{$\ddots$} & 0 \\
					\vdots & \vdots & \reflectbox{$\ddots$} & \reflectbox{$\ddots$} & \vdots \\
					f(\{(q - 1) r / q\}) & \phi(\{(q - 1) r / q\}) f(0) & 0 & \ldots & 0
				\end{array}\right|
			\]
			which is in turn equal to 
			\[
				(-1) ^ {(q - 2)(q - 1) / 2}f(0) ^ q
				\prod _ {i = 0} ^ {q - 2}
					\phi(\{(q - 1) r / q\}) \ldots \phi(\{(q - 1 - i) r / q\})
				\ne 0.
			\]
			So $f$ is cyclic for $T _ {\phi, \alpha}.$
		\end{proof}
	
		\begin{Def}
			Let $\phi \in L ^ \infty([0, 1]).$
			For every subset $A$ of $[0, 1]$, we define the set
			\[
				\text{Cycl} _ {A} ^ {(\phi)} =
				\bigcap _{\alpha \in A}
					\{
						f \in L ^ p([0, 1]) ;
						f \text{ is cyclic for } T _ {\phi, \alpha}
					\}
			\]
			of common cyclic vectors for all operators $T _ {\phi, \alpha}$, $\alpha \in A.$
			
			In particular $\text{Cycl} _ {\mathbf Q \cap (0, 1)} ^ {(\phi)}$ contains any holomorphic function $f$ on an open neighborhood of $[0, 1]$ such that $f(0) \ne 0$ if $\phi$ is an increasing
			function on $[0, 1]$ that is holomorphic on an open neighborhood of $[0, 1]$ satisfying $\phi(0) = 0.$
		\end{Def}
	
		It follows that Theorem \ref{comeager} can be generalized to the weighted translation operators.
		
		\begin{Lem}\label{continuouswt}
			Let $\phi \in L ^ \infty([0, 1])$.
			For every $Q \in \mathbf C[\xi]$, the map
			\[
				\fonction{
					\varphi _ Q \colon
				}{
					([0, 1], \abs{\cdot}) \times (L ^ p([0, 1]), \|\cdot\| _ p)
				}{
					(L ^ p([0, 1]), \|\cdot\| _ p)
				}{
					(\alpha, f)
				}{
					Q(T _ {\phi, \alpha}) f
				}
			\]
			is continuous.
		\end{Lem}
		\begin{proof}
			Let $Q = \sum _ {k = 0} ^ d a _ k \xi ^ k \in \mathbf C[\xi]$.
			As in the case of the Bishop operators, one can write 
			\[
				\varphi _ Q = \varphi _ {3, Q} \circ ((\varphi _ {2, Q}\circ \varphi _ 1) \oplus id)
			\]
			where 
			\[
				\fonction{
					id \colon
				}{
					(L ^ p([0, 1]), \| \cdot \| _ p)
				}{
					(L ^ p([0, 1]), \| \cdot \| _ p)
				}{
					f
				}{
					f,
				}
			\]
			\[
				\fonction{
					\varphi _ {1} \colon
				}{
					([0, 1], \abs{\cdot})
				}{
					(\mathcal B _ {\|\phi\| _ \infty}(L ^ p([0, 1])), SOT)
				}{
					\alpha
				}{
					T _ {\phi, \alpha},
				}
			\]
			\[
				\fonction{
					\varphi _ {2, Q} \colon
				}{
					(\mathcal B _ {\| \phi \| _ \infty}(L ^ p ([0, 1])), SOT)
				}{
					(\mathcal B _ {M _ Q}(L ^ p([0, 1])), SOT)
				}{
					T
				}{
					Q(T),
				}
			\]
			\[
				\fonction{
					\varphi _ {3, Q} \colon
				}{
					(\mathcal B _ {M _ Q}(L ^ p([0, 1])), SOT) \times (L ^ p([0, 1]), \|\cdot \| _ p)
				}{
					(L ^ p([0, 1]), \|\cdot\| _ p)
				}{
					(T, f)
				}{
					Tf
				}
			\]
			with $M _ Q = \sum _ {k = 0} ^ d \abs{a _ k} \|\phi \| _ \infty ^ k.$
			We will only prove the continuity of $\varphi _ 1$ since the other are easily proved.
			Let $\alpha \in [0, 1]$, and let $(\alpha _ n) _ {n \geq 1}$ te be a sequence of elements of $[0, 1]$ such that $\alpha _ n \to \alpha$ as $n \to + \infty$.
			Given $f \in L ^ p([0, 1])$, we are going to show that 
			$
				\| T _ {\phi, \alpha _ n} f - T _ {\phi, \alpha} f \| _ p 
				\to 0
			$ 
			as $n \to + \infty$.
			Let $\varepsilon > 0$.
			By density there exists $g \in \mathcal C([0, 1])$ such that $\|f - g\| _ p < \varepsilon.$
			For every $n \geq 1$,
			$
				\|T _ {\phi, \alpha _ n} f - T _ {\phi, \alpha} f \| _ p 
				\leq
				\|T _ {\phi, \alpha _ n}f - T _ {\phi, \alpha _ n} g \| _ p 
				+
				\|T _ {\phi, \alpha _ n} g - T _ {\phi, \alpha} g \| _ p
				+
				\|T _ {\phi, \alpha} g - T _ {\phi, \alpha} f\| _ p.
			$
			On the one hand for every $\beta \in [0, 1]$
			\[
				\| T _ {\phi, \beta} f - T _ {\phi, \beta} g \| _ p ^ p
				=
				\int _ 0 ^ 1
					\abs{
						\phi(t)f(\{t + \beta\}) - \phi(t)g(\{t + \beta\})
					} ^ p
					dt 
				\leq \| \phi \| _ \infty ^ p \|f - g \| _ p ^ p,	
			\]
			so that 
			$
				\|
					T _ {\phi, \alpha _ n} f - T _ {\phi, \alpha _ n} g 
				\| _ p 
				+
				\|
					T _ {\phi, \alpha} g - T _ {\phi, \alpha} f
				\| _ p 
				\leq
				2 \| \phi \| _ \infty \|f - g \| _ p.
			$
			On the other hand as in the proof of Lemma \ref{continuous}
			\begin{align*}
				\| T _ {\phi, \alpha _ n} g - T _ {\phi, \alpha} g \| _ p ^ p
				& =
				\int _ 0 ^ 1
					\abs{
						\phi(t) g(\{t + \alpha _ n\}) - \phi(t) g(\{t + \alpha\})
					} ^ p
					dt \\
				& \leq
				\| \phi \| _ \infty ^ p
				\int _ 0 ^ 1
					\abs{
						g(\{t + \alpha _ n\}) - g(\{t + \alpha\})
					} ^ p
					dt \\
				& \leq
				2 \| \phi \| _ \infty ^ p \varepsilon ^ p
			\end{align*}
			whenever $n$ is sufficiently large since $g$ is uniformly continuous and since $\alpha _ n \to \alpha$ as $n \to + \infty$.
			Eventually if $n$ is sufficiently large, we get that 
			$
				\|
					T _ {\phi, \alpha _ n} f - T _ {\phi, \alpha} f 
				\| _ p
				\leq 
				(2 + 2 ^ {1 / p}) \| \phi \| _ \infty \varepsilon
			$ 
			and thus $\varphi _ 1$ is continuous at $\alpha$.
		\end{proof}
		
		\begin{Thm}\label{comeagerw}
			Let $\phi \in \mathcal C([0, 1], \mathbf R)$ be an increasing
			convex
			function satisfying $\phi(0) = 0$ and which is holomorphic on an open neighborhood of $[0, 1].$
			The set of parameters
			$
				\{
					\alpha \in [0, 1] ;
					T _ {\phi, \alpha} \text{ is cyclic}
				\}
			$
			is a co-meager set in $[0, 1].$
		\end{Thm}
		\begin{proof}
			Let us set 
			$
				A = 
				\{
					(\alpha, f) \in [0, 1] \times L ^ p([0, 1]);
					f \text{ is cyclic for } T _ {\phi, \alpha}
				\}.
			$ 
			We will prove that this is a dense $G _ \delta$-set in $[0, 1] \times L ^ p([0, 1]).$
			Let $(U _ n) _ {n \geq 1}$ be a basis of open sets in $L ^ p([0, 1]).$
			For every $(\alpha, f) \in [0, 1] \times L ^ p([0, 1])$
			$f$ is cyclic for $T _ {\phi, \alpha}$ if and only if for every $n \geq 1$, there exists $Q \in \mathbf C[\xi]$ such that $Q(T _ {\phi, \alpha}) \in U _ n.$
			Thus 
			\[
				A = 
				\bigcap _ {n \geq 1}
					\bigcup _ {Q \in \mathbf C[\xi]}
						\varphi _ Q ^ {-1}(U _ n)
			\]
			where 
			the map $\varphi _ Q$ is defined on $[0, 1] \times L ^ p([0, 1])$ by $\varphi _ Q(\alpha, f) = Q(T _ {\phi, \alpha})f$ and is continuous by Lemma \ref{continuouswt} since $\phi$ is bounded on $[0, 1].$
			So $A$ is a $G _ \delta$-set in $[0, 1] \times L ^ p([0, 1]).$ 
			Let now $\varepsilon > 0$ and $(\beta, g) \in [0, 1] \times L ^ p([0, 1]).$
			There exists $\alpha \in (0, 1) \cap \mathbf Q$ such that $\abs{\beta - \alpha} < \varepsilon.$
			Since $T _ {\phi, \alpha}$ is cyclic by Theorem \ref{cyclicityw} and since $\sigma _ p(T _ {\phi, \alpha} ^ *) = \emptyset$ by Theorem \ref{wtspectrum}, the set of cyclic vectors for $T _ {\phi, \alpha}$ is a dense $G _ \delta$-set in $L ^ p ([0, 1]).$
			Then there exists a cyclic vector $f$ for $T _ {\phi, \alpha}$ such that $\|g - f\| _ p < \varepsilon.$
			So $A$ is a dense $G _ \delta$-set in $[0, 1] \times L ^ p([0, 1]).$
			
			By the Kuratowski-Ulam's Theorem, the set 
			\begin{align*}
				B 
				& = 
				\{
					\alpha \in [0, 1] ; 
					\{
						f \in L ^ p([0, 1]) ;
						f \text{ is cyclic for } T _ {\phi, \alpha}
					\}
					\text{ is co-meager set in }
					L ^ p([0, 1])
				\}
			\end{align*}
			is a co-meager set in $[0, 1]$ and $B \subset \{\alpha \in [0, 1] ; T _ {\phi, \alpha} \text{ is cyclic}\}$.
		\end{proof}
	
		The function $\varphi _ {P, f} \colon \alpha \mapsto P(T _ {\phi, \alpha})f$ still being continuous, one can extend Theorem \ref{psi} for a weighted translation operator $T _ {\phi, \alpha}$ whenever $\alpha$ is an irrational number in $[0, 1]$ sufficiently well approached by its convergents.
		
		\begin{Thm}\label{psiwt}
			Let $\phi \in \mathcal C([0, 1], \mathbf R)$ be an increasing 
			function satisfying $\phi(0) = 0$ and which is holomorphic on an open neighborhood of $[0, 1].$
			Let $f \in \text{Cycl} _ {\mathbf Q \cap (0, 1)} ^ {(\phi)}$ be a common cyclic vector for $T _ {\phi, \alpha}$ for every $\alpha \in (0, 1) \cap \mathbf Q$.
			There exists a function $\psi _ {\phi, f} \colon \mathbf N \to \mathbf R _ +$ with the following property: if $(p _ n / q _ n) _ {n \geq 0}$ are the convergents of an irrational number $\alpha$ in $[0, 1]$ and if 
			for every $n \geq 0$ there exists $n _ 0 \geq n$ such that 
			$
				q _ {n _ 0 + 1} > \psi _ {\phi, f}(q _ {n _ 0}),
			$
			then $f$ is cyclic for $T _ {\phi, \alpha}$.
		\end{Thm}
		\begin{proof}
			Since the weighted translation operator $T _ {\phi, \alpha}$ is cyclic whenever $\alpha \in (0, 1) \cap \mathbf Q$ and since the map $\varphi _ {Q, f} \colon \alpha \mapsto Q(T _ {\phi, \alpha})f$ is continuous whenever $Q \in \mathbf C[\xi]$, the result follows from the proof of Theorem \ref{psi}.
		\end{proof}
	
		\begin{Rmq}
			Since we don't assume in Theorem \ref{psiwt} that $\phi$ is a convex function, the operator $T _ {\phi, \alpha} ^ *$ may have eigenvalues according to Theorem \ref{wtspectrum}.
			Then we don't know if the set of cyclic functions for $T _ {\phi, \alpha}$ is a dense set in $L ^ p([0, 1])$ whenever $T _ {\phi, \alpha}$ is cyclic.
		\end{Rmq}	

	\section{Open questions}
		We present in this last section some further comments and open questions. \\
		
		To begin with, although we know that the weighted translation operators $T _ {\phi, \alpha}$ do not satisfy the Hypercyclicity Criterion, we are unable to prove that they are not hypercyclic.
		
		\begin{Ques}
			For every $\phi \in L ^ \infty([0, 1])$ and every $\alpha \in [0, 1]$, is it true that $T _ {\phi, \alpha}$ is not hypercyclic on $L ^ p([0, 1])$? 
		\end{Ques}
	
		As in the hypercyclic case, we do not know if the weighted translation operators $T _ {\phi, \alpha}$ are not supercyclic in general, even if we do know that they cannot satisfy the Supercyclicity Criterion.
		
		\begin{Ques}
			For every $\phi \in L ^ \infty([0, 1])$ and every $\alpha \in [0, 1]$, is it true that $T _ {\phi, \alpha}$ is not supercyclic on $L ^ p([0, 1])$? 
		\end{Ques}
			
		Our results on cyclicity of operators $T _ \alpha$ for $\alpha \in \mathbf R \backslash \mathbf Q$ give rise to many questions.
		
		\begin{Ques}
			Given $f \in L ^ p([0, 1]) \backslash \{\mathbf 0\}$, does there exist a finite set $F \subset \mathbf Q$ such that $f$ is cyclic for $T _ \alpha$ for every $\alpha \in \mathbf Q \backslash F$?
		\end{Ques}
				
				Thanks to Theorem \ref{psi}, this set could help to explicit a set of common cyclic vectors for $T _ \alpha$ for a large set of $\alpha \in [0, 1].$
				
				It would also be interesting to take a closer look at the function $\psi$ appearing in the statement of Theorem \ref{psi}.
				
				\begin{Ques}
					Can the function $\psi$ appearing in the statement of Theorem \ref{psi} be made explicit, for instance in the case where $f = \mathbf 1$?
				\end{Ques}
			
				We know that $f = \mathbf 1$ is a common cyclic vector for $T _ \alpha$ whenever $\alpha \in (0, 1) \cap \mathbf Q$.
				Let $\varepsilon > 0$ and let $\alpha = r / q$ be a rational number in $(0, 1).$
				In order to obtain an explicit form for $\psi$, we would need for every $g \in L ^ \infty([0, 1])$ vanishing on a set $\Omega _ {n, f}$ to explicit a polynomial $Q \in \mathbf C[\xi]$ such that 
				$\| Q ( T _ \alpha) \mathbf 1 - g \| _ p < \varepsilon.$
				To do so, writing $g = g _ 0 T _ \alpha ^ 0 f + \ldots + g _ {q - 1} T _ \alpha ^ {q - 1} f$, we would have to take a look at polynomials $Q _ j \in \mathbf C[\xi]$ such that $\|Q _ j - g _ j\| _ \infty < \varepsilon$ whenever $j \in \{0, \ldots, q - 1\}.$
				In the study of such polynomials, an obstacle is the study of the determinant 
				\[
					\Delta(\mathbf 1, r / q)(t)
					=
					\begin{vmatrix}
						1 & t & \ldots & t \{t + r / q\} \ldots \{t + (q - 2) r / q\} \\
						1 & \{t + r / q\} & \ldots & \{t + r / q\} \ldots \{t + (q - 1) r / q\} \\
						\vdots & \vdots & & \vdots \\
						1 & \{t + (q - 1) r / q\} & \ldots & \{t + (q - 1) r / q\} t \ldots \{t + (q - 3) r / q\}
					\end{vmatrix}.
				\]
				This determinant seems to never vanish on $[0, 1 / q)$ and to be a monotone function on $[0, 1 / q).$ 
				Proving such properties should lead to estimates 
				allowing to obtain an explicit form of the function $\psi.$ \\
				
				Getting back to the study of invariant subspaces, we recall that in 
				\cite{cha20}, the authors have proved that $T _ \alpha$ admits a non-trivial closed hyperinvariant subspace in $L ^ p([0, 1])$ as soon as the convergents $(p _ n / q _ n) _ {n  \geq 0}$ of $\alpha$ satisfy 
				\[
					\log(q _ {n + 1}) 
					\mathop{=} 
						 \limits _ {
						 	n \to + \infty
						 }
						O 
						\left(
							\frac{q _ n}{\log(q _ n) ^ 3}
						\right).
				\]
				This condition gives a bound on the growth of the sequence $(q _ n) _ {n \geq 0}$ of denominators of the convergents of $\alpha.$
				On the other hand, the results we proved in Theorem \ref{psi} state that $T _ \alpha$ will be cyclic as soon as the sequence 
				$
					(q _ n) _ {n \geq 0}
				$ 
				has infinitely many sufficiently large gaps.
				This is not surprising.
				Indeed cyclicity and admiting  non-trivial closed invariant subspaces can be seen as opposed properties of an operator $T$ since $T$ does not admit such a subspace if and only if every non-zero vector is cyclic for $T.$
				Nevertheless, an optimization of the function in the statement of Theorem \ref{psi} $\psi$ could possibly lead to a positive answer to the following question.
				
				\begin{Ques}
					Is it possible to explicit real numbers $\alpha$ such that $T _ \alpha$ is cyclic and admits a non-trivial (hyper)invariant subspace? 
				\end{Ques}


\newpage
	
	\begin{bibdiv}
		\begin{biblist}
			
			\bib{bay09}{book}{
				AUTHOR = {F. Bayart and \'E. Matheron},
				TITLE = {Dynamics of linear operators},
				SERIES = {Cambridge Tracts in Mathematics},
				VOLUME = {179}
				PUBLISHER = {Cambridge University Press},
				YEAR = {2009},
			}
		
			\bib{ber04}{article}{
				AUTHOR = {{T. Berm\'udez, A. Bonilla} and A. Peris},
				TITLE = {On hypercyclicity and supercyclicity criteria},
				JOURNAL = {Bull. Austral. Math. Soc.},
				FJOURNAL = {Bulletin of the Australian Mathematical Society},
				VOLUME = {70},
				YEAR = {2004},
				PAGES = {45 - 54},				
			}

			\bib{bes99}{article}{
				AUTHOR = {J. B\`es and A. Peris},
				TITLE = {Hereditarily hypercyclic operators},
				JOURNAL = {J. Funct. Anal.},
				FJOURNAL = {Journal of Functional Analysis},
				VOLUME = {167},
				YEAR = {1999},
				PAGES = {94 - 112},
			}
			
			\bib{bug04}{book}{
				AUTHOR = {Y. Bugeaud},
				TITLE = {Approximation by algebraic numbers},
				SERIES = {Cambridge Tracts in Mathematics},
				VOLUME = {160},
				PUBLISHER = {Cambridge University Press},
				YEAR = {2004},
			}
			
			\bib{cha11}{book}{
				AUTHOR = {I. Chalendar and J. R. Partington},
				TITLE = {Modern approaches to the invariant-subspace problem},
				SERIES = {Cambridge Tracts in Mathematics},
				VOLUME = {188},
				PUBLISHER = {Cambridge University Press},
				YEAR = {2011},
			}

			\bib{cha10}{book}{
				AUTHOR = {{I. Chalendar, J. R. Partington} and E. Pozzi},
				TITLE = {
					Multivariable weighted composition operators: lack of point spectrum, and cyclic vectors. 
					In: Topics in operator theory
				},
				SERIES = {Operator theory: advances and applications},
				VOLUME = {202},
				PUBLISHER = {Birkh\"auser Basel},
				YEAR = {2010},
			}
		
			\bib{cha20}{article}{
				AUTHOR = {{F. Chamizo, E. A. Gallardo-Guti\'errez, M. Monsalve-L\'opez} and A. Ubis},
				TITLE = {Invariant subspaces for Bishop operators and beyond},
				JOURNAL = {Adv. Math.},
				FJOURNAL = {Advances in Mathematics},
				VOLUME = {375},
				YEAR = {2020},
			}
		
			\bib{dav74}{article}{
				AUTHOR = {A. M. Davie},
				TITLE = {Invariant subspaces for Bishop's operators},
				JOURNAL = {Bull. Lond. Math. Soc.},
				FJOURNAL = {Bulletin of the London Mathematical Society},
				VOLUME = {6},
				YEAR = {1974},
				PAGES = {343 - 348},
			}

			\bib{fla08}{article}{
				AUTHOR = {A. Flattot},
				TITLE = {Hyperinvariant subspaces for Bishop-type operators},
				JOURNAL = {Acta Sci. Math.},
				FJOURNAL = {Acta Scientiarum Mathematicarum},
				VOLUME = {74},
				YEAR = {2008},
				PAGES = {689 - 718},
			}
			
			\bib{get87}{article}{
				AUTHOR = {R. M. Gethner and J. H. Shapiro},
				TITLE = {Universal vectors for operators on spaces of holomorphic functions},
				JOURNAL = {Proc. Amer. Math. Soc.},
				FJOURNAL = {Proceedings of the American Mathematical Society},
				VOLUME = {100},
				YEAR = {1987},
				NUMBER = {2},
				PAGES = {281 - 288},
			}
		
			\bib{gro11}{book}{
				AUTHOR = {K.-G. Grosse-Erdmann and A. Peris Manguillot},
				TITLE = {Linear chaos},
				SERIES = {Universitext},
				PUBLISHER = {Springer London},
				YEAR = {2011},
			}
			
			\bib{kec95}{book}{
				AUTHOR = {A. S. Kechris},
				TITLE = {Classical descriptive set theory},
				SERIES = {Graduate Texts in Mathematics},
				PUBLISHER = {Springer New York, NY},
				YEAR = {1995},
			}
		
			\bib{leo04}{article}{
				AUTHOR = {F. Le\'on-Saavedra and V. M\"uller},
				TITLE = {Rotations of hypercyclic and supercyclic operators},
				JOURNAL = {Integr. Equ. Oper. Theory},
				FJOURNAL = {Integral Equations and Operator Theory},
				VOLUME = {50},
				YEAR = {2004},
				PAGES = {385 - 391},
			}
		
			\bib{mac90}{article}{
				AUTHOR = {G. W. MacDonald},
				TITLE = {Invariant subspaces for Bishop-type operators},
				JOURNAL = {J. Funct. Anal.},
				FJOURNAL = {Journal of Functional Analysis},
				VOLUME = {91},
				YEAR = {1990},
				PAGES = {287 - 311}
			}

			\bib{par65}{article}{
				AUTHOR = {S. K. Parrott},
				TITLE = {Weighted translation operators},
				JOURNAL = {University of Michigan, Ph.D.},
				FJOURNAL = {The University of Michigan, Ph.D.},
				YEAR = {1965},
			}
			
		\end{biblist}
	\end{bibdiv}
\end{document}